\documentclass[twoside,11pt]{article}

\usepackage{amsmath}
\usepackage{amssymb}
\usepackage{latexsym,amsthm}
\usepackage{mathtools}
\usepackage[final]{showkeys}
\usepackage{graphicx}
\usepackage{mathrsfs}
\usepackage{xcolor}
\usepackage{hyperref}
\usepackage{enumitem}

\usepackage{amsmath}

\newtheorem{lemma}{Lemma}
\newtheorem{corollary}{Corollary}
\newtheorem{theorem}{Theorem}

\newtheorem{proposition}{Proposition}
\theoremstyle{definition}
\newtheorem{remark}{Remark}

\newcommand{\diff}{\mathop{}\!\mathrm{d}}

\newcommand{\cC}{\mathscr{C}}

\title{
Boundary Value Problems for a Special Helfrich Functional for 
\\Surfaces of Revolution\\
-- Existence and asymptotic behaviour --
}
\author{Klaus Deckelnick\footnote{Fakult\"at f\"ur Mathematik, Otto-von-Guericke-Universit\"at Magdeburg, Postfach 4120, D-39016 Magdeburg, Germany} \and Marco Doemeland\footnote{Math 111810, RWTH Aachen, D-52056 Aachen, Germany} \and Hans-Christoph Grunau{\footnotesize *}
}
\date{\today}

\begin{document}

\maketitle

\begin{abstract}
	The central object of this article is (a special version of)
	the Helfrich functional which is the sum of the Willmore  functional and the area functional times a weight factor $\varepsilon\ge 0$. We collect several results concerning the existence of solutions to a Dirichlet boundary value problem for Helfrich surfaces of revolution and cover some specific regimes of boundary conditions and weight factors $\varepsilon\ge 0$. These results are obtained with the help of different techniques like an energy method, gluing techniques and the use of the implicit function theorem close to Helfrich cylinders. In particular, concerning the regime of boundary values, where a catenoid exists as a global minimiser of the area functional, existence of minimisers of the Helfrich functional is established for \emph{all} weight factors $\varepsilon\ge 0$. For the singular limit of weight factors $ \varepsilon \nearrow \infty  $ they  converge uniformly to the catenoid which minimises the surface area in the class of surfaces of revolution.
\end{abstract}

\section{Introduction}

The Helfrich energy for a sufficiently  smooth (two dimensional) surface $ S \subset \mathbb{R} ^3 $ (with or without boundary), introduced by Helfrich in \cite{Helfrich} and Canham in \cite{Canham}, is defined by
$$
{\mathscr H}(S) := \int_{S} (H-H_0)^2 \,  \diff S - \gamma \int_S K \,  \diff S + \varepsilon \int_{S} \, \diff S \, .
$$
Here the integration is done with respect to the usual $ 2 $-dimensional area measure $ \diff S $, $ H $ is the mean curvature of $ S $, i.\,e. the mean value of the principal curvatures, $ K $ is the Gaussian curvature, $ \gamma  \in \mathbb {R} $ is a constant bending rigidity, $ \varepsilon  \ge 0$ the weight factor of the area functional and $ H_0 \in \mathbb {R}  $ a given spontaneous curvature. For simplicity we will always assume $ H_0=0 $ so that the first term in the above energy reduces to the well--known Willmore functional $\int_S H^2 \, \diff S$. We expect nontrivial $H_0\not=0$ to result in completely different geometric phenomena; we think that our methods cannot directly be extended to this case. 
For the investigation of closed surfaces in the case $H_0\not=0$ one may see  e.g. \cite{BernardWheeler,ChoksiVeneroni}.
We also choose $ \gamma =0 $ since $ \int_ {S}K \, \diff S $ in the second term will only contribute to boundary terms by the Gauss-Bonnet theorem, which are constant thanks to the given boundary conditions. The Euler-Lagrange equation to the resulting Helfrich energy is known as the Helfrich equation and is given by
\begin{align*}
	\Delta  _ { S } H  + 2H \, (H ^2- K) - 2 \varepsilon  H  \, = \,  0 \qquad  \mbox{ on }   S ,
\end{align*}
where $ \Delta _S $ denotes the Laplace-Beltrami operator on $ S $.

There are several applications of the Helfrich energy in e.\,g. biology in modelling red blood cells or lipid bilayers (see e.\,g. \cite{Canham,Helfrich,Nitsche,Ou-Yang}) and in modelling thin elastic plates (see e.\,g. \cite{Germain}). As noted in \cite{dLPR,Nitsche}, the Willmore functional 
was already considered in the early 19th century (see \cite{Poisson,Germain})
and has been periodically reintroduced according to the availability of more
powerful mathematical tools e.g. in the early 1920s (see \cite{Thomsen}) 
and most successfully by Willmore (see  \cite{Wil65}) in the 1960s.
So far, mathematical research has mainly focused on closed (i.e. compact without boundary) surfaces minimising  the Willmore functional, see e.\,g. \cite{Simon,BauerKuwert,MarquesNeves} and references therein. Some results on closed surfaces minimising the Helfrich energy under fixed area and enclosed volume for the axisymmetric case can be found in \cite{ChoksiVeneroni}.
In \cite{CMV} this was generalised to a multiphase Helfrich energy, while recently in \cite{BLS} the same problem was studied in the setting of curvature varifolds, thereby dropping the symmetry assumptions.

In contrast to closed surfaces we are here interested in a Dirichlet boundary value problem for minimising the Helfrich energy. Some existence results on the Dirichlet problem for the Willmore functional can be found e.\,g. in \cite{Schaetzle} for a class of branched immersions in $\mathbb{R}^3\cup \{\infty\}$. 
The Douglas (or Navier) boundary value problem was studied in \cite{NovagaPozzetta,Pozzetta}.
In the class of surfaces of revolution existence results for the Willmore functional were obtained in \cite{DDG,DFGS,EichGr}; see also references therein. In \cite{DGR} existence of minimisers of a relaxed Willmore functional in the class of graphs over two-dimensional domains was proved. For 
the Helfrich functional an existence result for branched immersions in $\mathbb{R}^3$ was found in \cite{Eichmann}. This is somehow related to \cite{dLPR}, where an area constraint was imposed in order to minimise the Willmore functional. 
 \\ \-

 In this paper we focus on surfaces of revolution $ S $
\begin{align*}
	(x,\theta)\mapsto  \big(x,u(x)\cos \theta, u(x)\sin \theta \big) \, ,
\quad  x\in[-1,1],~\theta\in[0,2\pi],
\end{align*}
for some sufficiently smooth profile curve $ u \colon [-1,1] \rightarrow (0, \infty ) $. We then consider the Dirichlet boundary value problem for Helfrich surfaces of revolution
\begin{align}
	\label{eq:helfrich_dbvp}
	\left\{~
	\begin{aligned}
		\Delta  _ { S } H  + 2H \, (H ^2- K) - 2 \varepsilon  H  \, &= \,  0 \qquad  \mbox{ in }   (-1,1), \\
		u(\pm 1) \,=\, \alpha, \quad  u'(\pm 1)\, & = \, 0 \, ,
	\end{aligned}
	\right.
\end{align}
for a given boundary value $ \alpha >0 $ and for profile curves $ u $ that are even, i.\,e. $ u(x) = u(-x) $. The precise setting and a derivation of (\ref{eq:helfrich_dbvp}) will be given in Section \ref{helfrichfunctional}. 
First existence results for (\ref{eq:helfrich_dbvp}) have been obtained in \cite{Scholtes} and \cite{Doemeland}. The goal of this paper is to  extend these
results in several directions including a description of  the qualitative behaviour of solutions when possible.
Depending on the two parameters $ \alpha  $ and $ \varepsilon  $, Figure \ref{fig:domains_of_existence}
\begin{figure}[h]
	\centering
	\definecolor{red}{RGB}{210,34,16}
	\definecolor{blue}{RGB}{31,96,205}
	\definecolor{green}{RGB}{6,151,31}
	\definecolor{orange}{RGB}{251,123,22}
	\definecolor{light blue}{RGB}{25,212,221}
	\definecolor{purple}{RGB}{122,68,179}
	\newcommand{\domaincolor}[1]{\textbf{\textcolor{#1}{\underline{#1}}}}
	\includegraphics[]{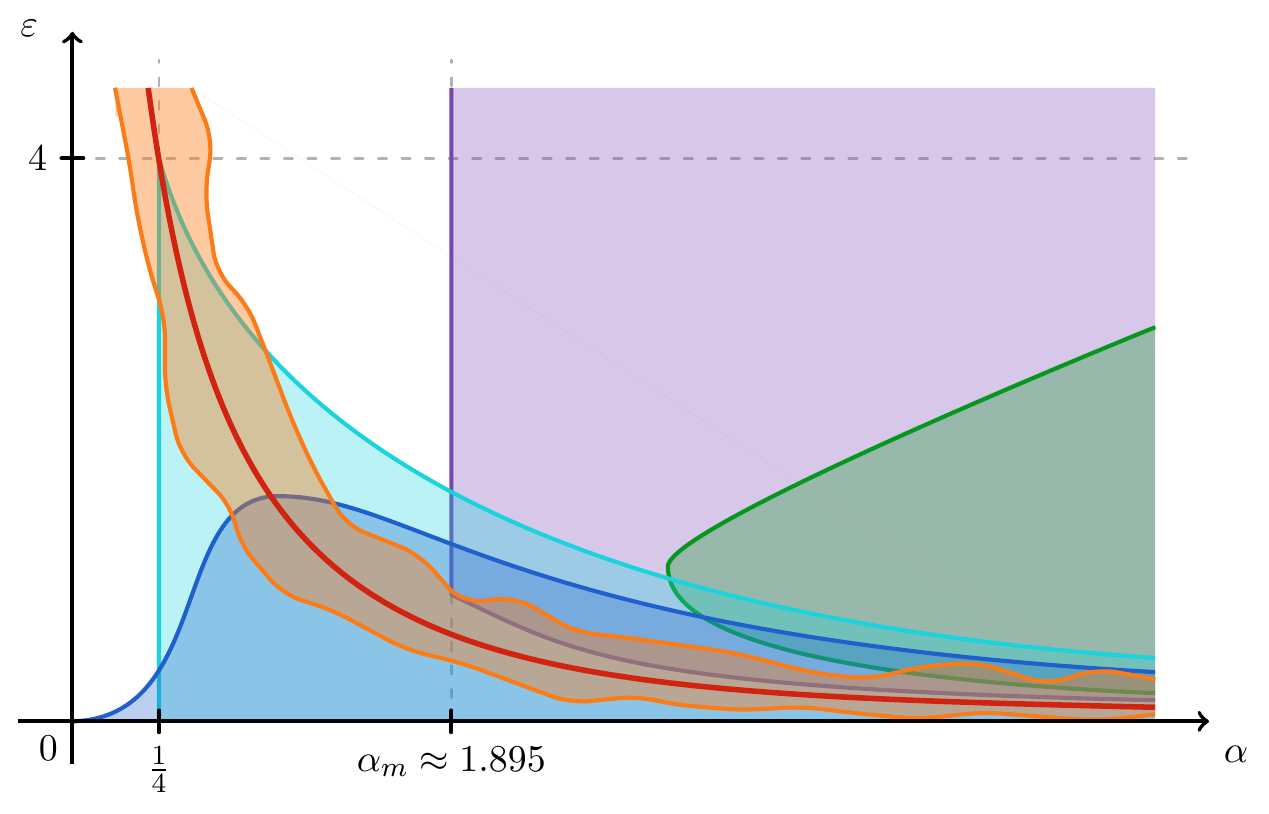}
	\caption{A sketch of domains for existence of solutions to the Dirichlet problem \eqref{eq:helfrich_dbvp} depending on $ (\alpha ,\varepsilon ) $: The domains \domaincolor{blue} and \domaincolor{light blue} show existence of Helfrich minimisers  obtained in Theorem \ref{thm:existence_energy_method} and \cite[Satz\;5.3.3\,and\;5.3.4]{Doemeland}. The \domaincolor{red} curve where $ \varepsilon = \varepsilon _ \alpha = \frac{1}{4 \alpha ^2} $ displays the Helfrich cylinders (see \cite[Lemma\;4.1]{Scholtes}). The \domaincolor{orange} domain sketches existence of solutions obtained by perturbing the Helfrich cylinders (see Theorem \ref{thm:exist_implicit_fct}).  The \domaincolor{purple} domain is obtained with the help of gluing techniques (see Theorem \ref{thm:existence_conv_to_cat}) and may be compared to the \domaincolor{green} domain implicitly found in \cite[Theorem\;4.14]{Scholtes}.}
	\label{fig:domains_of_existence}
\end{figure}
gives an overview of domains on which existence of solutions to the Helfrich problem \eqref{eq:helfrich_dbvp} is guaranteed. Let us describe these domains in more detail and relate them to
the various sections of this paper. 

As a starting point we consider in Section \ref{subsec:exist_energy_bounds}  the Helfrich boundary value problem as a \emph{regular} perturbation of the Willmore
boundary value problem for small  $\varepsilon$. Note however that a direct application of the techniques developed in  \cite{DDG} for the Willmore functional is not possible because they are heavily based on its conformal invariance and on explicitly known Willmore minimisers. Nevertheless, by applying the direct method  and  utilising an energy gap of the Willmore energy to $ 4 \pi  $ we find in Theorem~\ref{thm:existence_energy_method} solutions of \eqref{eq:helfrich_dbvp} for any 
$\alpha>0$ and $\varepsilon\ge 0$ not too large, see the blue regions in
Figure~\ref{fig:domains_of_existence}. In this case we solely give a quantitative existence proof with no qualitative description of the minimisers. 

In order to develop some intuition for what happens when $\varepsilon$ increases we then concentrate in Section~\ref{sec:perturbation_cylinder}  on the local picture around a Helfrich cylinder, i.e. 
 $ u(x) \equiv \alpha  $, which 
is an explicit solution and the unique Helfrich minimiser in the case $ \varepsilon = \varepsilon _ \alpha := \frac{1}{4 \alpha ^2} $ (see \cite[Lemma\;4.1]{Scholtes}; red curve in Figure \ref{fig:domains_of_existence}). 
Linearising (\ref{eq:helfrich_dbvp}) for fixed $ \alpha  $  at the corresponding Helfrich 
cylinder, the implicit  function theorem yields the  existence of a smooth 
family $(u_ \varepsilon )_{\varepsilon\approx \varepsilon _ \alpha }$ of solutions of  \eqref{eq:helfrich_dbvp}, see Theorem \ref{thm:exist_implicit_fct}; orange domain in Figure \ref{fig:domains_of_existence}. As a first step
in order to understand whether and how this local branch  is part of a 
possibly existing family of solutions $(u_ \varepsilon )_{\varepsilon\ge 0}$ we consider 
the rate of change function $\mathit{rc}_\alpha:=  \frac{\partial u_ \varepsilon }{\partial \varepsilon } |_ {\varepsilon =  \varepsilon _ \alpha } $, for which we derive an explicit formula. A careful analysis of the resulting
expression shows that $\mathit{rc}_\alpha$ is negative, see Theorem \ref{thm:change_of_rate_negative}, which 
gives rise to the (open) conjecture that the family 
$(u_ \varepsilon )_{\varepsilon\in [0,\varepsilon _ \alpha]}$ exists and is
decreasing with respect to $\varepsilon$. Furthermore, we find that the shape of $\mathit{rc}_\alpha$ strongly depends  on the boundary value $ \alpha   $. More precisely, there exists a critical value $\alpha _ \text{crit} \approx 0.18008$
such that for $ \alpha > \alpha _ \text{crit}$ the function $\mathit{rc}_\alpha$ is strictly increasing on $ [0,1] $ (see Theorem \ref{thm:change_of_rate_oscillation}), while for $ \alpha < \alpha _ \text{crit}   $ and $ \alpha \searrow 0 $ an increasing number of oscillations around a negative value show up. This interesting phenomenon 
of ``overshoot'' and subsequent oscillations around an expected ``attracting state'' has often been observed in numerical experiments and to our knowledge this is the first time where an analytical proof is given. It indicates that 
for small $\alpha>0$, understanding existence, qualitative properties and 
asymptotic behaviour for the full range of $\varepsilon\in[0,\infty)$ remains 
an interesting and  challenging open problem.

Next, in Section~~\ref{subsec:exist_gluing} we consider the case of larger values of $\varepsilon$. This  means that the area functional as part of the Helfrich energy is given a bigger weight suggesting that minimisers of the area functional in the class of surfaces of
revolution might be useful when minimising the Helfrich funtional in this case. It is well--known that there exists a threshold
$ \alpha _m \approx 1.805 $ such that  for $ \alpha > \alpha _m $ the minimiser is given by a catenoid while for $ \alpha < \alpha _m $ it is a so-called Goldschmidt solution. Using this observation in the case $\alpha \geq \alpha_m$ 
we are able  to decrease the energy of minimising sequences for the Helfrich functional by gluing in suitable catenoids leading to the existence of minimisers, see Theorem \ref{thm:existence_conv_to_cat} and the
purple region in Figure \ref{fig:domains_of_existence}.
At the same time this process gives some qualitative information about these minimisers. In addition, we prove in Corollary \ref{cor:existence_conv_to_cat} that the sequence of  minimisers converges to the area-minimising catenoid
in the singular limit  $ \varepsilon \nearrow \infty  $. So far, there is only numerical evidence that 
for $0< \alpha < \alpha _m$ minimisers may exist for any $\varepsilon\ge 0$
and converge in some singular sense to the Goldschmidt minimal surface as $ \varepsilon \nearrow \infty  $.

In Section \ref{sec:summary} we finally summarise our results and describe some open problems. At the end we add two appendices. Appendix~\ref{sec:divergence_form_helfrich} proves a divergence form of the Helfrich equation and Appendix~\ref{sec:appendix_est_oscill} collects the basic estimates to prove the theorems of Section~\ref{sec:perturbation_cylinder}.


\section{Geometric background for surfaces of revolution}
\label{helfrichfunctional}

For a sufficiently smooth (two dimensional) surface $S \subset \mathbb{R}^3$ 
(with or without boundary)
we consider the Helfrich functional (see e.g. \cite{Helfrich,Nitsche}) 
$$
{\mathscr H}(S) := \int_{S} H^2 \, \diff S - \gamma \int_S K \, \diff S + \varepsilon \int_{S} \, \diff S \, .
$$
Here $\diff  S$ denotes integration with respect to the surface area measure, while
$$
H(x):= \frac{1}{2} \bigl( \kappa_1 (x) + \kappa_2 (x) \bigr) \quad \mbox{ and } \quad  K(x) := \kappa_1(x) \kappa_2(x)
$$
are the mean curvature and the Gauss curvature of $S$ respectively. 
Furthermore, $\gamma \in \mathbb{R}$ as well as $\varepsilon \geq 0$ are given constants. In what follows we consider surfaces of revolution $S$
\begin{align*}
	(x,\theta)\mapsto  \big(x,u(x)\cos \theta, u(x)\sin \theta \big) \, , \quad  x\in[-1,1],~\theta\in[0,2\pi],
\end{align*}
which arise when the sufficiently smooth profile curve
$
u:[-1,1]\to (0,\infty)
$
is rotated around the $x$-axis. 
Expressing $H$ and $K$ in terms of $u$ yields
\begin{align*}
H(x) &= \frac{1}{2} \Bigg(  \frac{1}{u(x)\, \sqrt{ 1+ u'(x)^2 } } - \frac{u''(x)}{ {(1+u'(x)^2)}^{3/2}} \Bigg) \, , \\
K(x) &= -\dfrac{u''(x)}{u(x) {(1 + u'(x)^2)}^2} \, .  
\end{align*}
Since $\diff S =  u(x) \sqrt{1+u'(x)^2}\, \diff x\, \diff \theta$ we have for the surface area
\begin{equation} \label{surfacearea}
\mathscr A(u)\, := \, 2 \pi \int_{-1}^1 u(x) \sqrt{1+ u'(x)^2} \; \diff x
\end{equation}
while the Willmore functional is given by
\begin{equation}  \label{def:Willmore_rev}
{\mathscr W}(u) \,:=\,  \dfrac{\pi}{2} \int^1_{-1} \left( \dfrac{1}{u(x) \sqrt{1+u'(x)^2}} - \dfrac{u''(x)}{{(1+u'(x)^2)}^{3/2}} \right)^2 u(x) \sqrt{1 + u'(x)^2} \; \diff x \, .
\end{equation}
In what follows we shall consider for $\alpha>0$  the following class of admissible functions:
\begin{equation}\label{eq:adm_fctns}
N_{\alpha}:=\{ u\in H^2 (-1,1)\, :\, u \mbox{ is even}, u>0 \mbox{ in } [-1,1],  u(\pm 1)=\alpha,u'(\pm 1)=0\}.
\end{equation} 
Note  that for $u \in N_\alpha$
\begin{equation}  \label{intK}
\int_S K \, \diff  S \, = \,  - 2 \pi \int_{-1}^1  \frac{u''(x)}{{( 1+ u'(x)^2 )}^{3/2}} \, \diff x 
\, = \,  - 2 \pi \left[ \frac{u'(x)}{\sqrt{1+ u'(x)^2}} \right]^{x=1}_{x=-1} \, = \, 0 \, ,
\end{equation}
so that in our setting the Helfrich functional $\mathscr H = \mathscr H_\varepsilon: N_\alpha \rightarrow \mathbb{R}$ takes the form
\begin{equation} \label{def:Helfrich_rev}
{\mathscr H}_\varepsilon(u) := \mathscr W(u) + \varepsilon \mathscr A(u).
\end{equation}
Using the fact that  $\frac{u''}{(1+(u')^2)^{3/2}}= \frac{\diff }{\diff x} \bigl( \frac{u'}{\sqrt{1+(u')^2}} \bigr)$ we find that
\begin{eqnarray}
\lefteqn{ \hspace{1cm} \int_a^b \Bigl( \frac{1}{ u(x)\,\sqrt{ 1+ u'(x)^2 } } -  \frac{u''(x)}{{(1+ u'(x)^2)}^{3/2}} \Bigr)^2 u(x) \sqrt{1+u'(x)^2} \, \diff x } \label{equiv}  \\
&  = & \   \int_a^b \left[ \bigl( \frac{1}{u(x) \sqrt{1+u'(x)^2}} \bigr)^2 + \bigl( \frac{u''(x)}{{(1+u'(x)^2)}^{3/2}} \bigr)^2 \right] u(x) \sqrt{1+u'(x)^2} \, \diff x
- 2 \frac{u'(x)}{\sqrt{1+u'(x)^2}} \Big|^b_a  \nonumber   \\
&  = &  \int_a^b \Bigl( \frac{1}{ u(x)\,\sqrt{ 1+ u'(x)^2 } } +  \frac{u''(x)}{{(1+ u'(x)^2)}^{3/2}} \Bigr)^2 u(x) \sqrt{1+u'(x)^2}  \,  \diff x - 4 \frac{u'(x)}{\sqrt{1+u'(x)^2}} \Big|^b_a. \nonumber
\end{eqnarray}
In particular we have for $u \in N_\alpha$ that
\begin{equation} \label{eq:helfrich_functional1}
 \mathscr H_\varepsilon(u) =  \frac{\pi}{2} \int_ {-1}^1  \Bigg(\frac{1}{ u(x)\,\sqrt{ 1+ u'(x)^2 } } + 4 \varepsilon \, u(x)\,\sqrt{ 1+ {u'(x)}^2 }\Bigg)\, \diff x   + \frac{\pi }{2} \int_ {-1}^1  
 \frac{u(x) u''(x)^2}{{( 1+ u'(x)^2 )}^{5/2}} \, \diff x. 
\end{equation} 
%
Let us next consider the Euler--Lagrange equation for $\mathscr H_\varepsilon$ and fix $u \in N_\alpha \cap C^4([-1,1])$. Then we calculate for the first variation of
$\mathscr A$ at $u$ in direction $\varphi \in H^2_0(-1,1)$
\begin{equation} \label{fva}
\langle \mathscr A'(u), \varphi \rangle = 2 \pi \int_{-1}^1 \left[\varphi(x) \sqrt{1+u'(x)^2} + \frac{u(x) u'(x) \varphi'(x)}{\sqrt{1+ u'(x)^2}} \right] \, \diff x = 4 \pi \int_{-1}^1 u(x) \varphi(x) H(x) \, \diff x
\end{equation}
while \cite[Lemma A.1]{DeckelnickGrunau3} yields
\begin{equation} \label{fvw}
\langle \mathscr W'(u),\varphi \rangle = - 2 \pi \int_{-1}^1 u(x) \varphi(x) \bigl( \Delta_S H + 2 H(H^2 -K) \bigr)(x) \, \diff x.
\end{equation}
Combining (\ref{fva}) and (\ref{fvw}) we see that if $u \in N_\alpha \cap C^4([-1,1])$ is a critical point of the
 energy $\mathscr H_\varepsilon$ then it is a solution of  Helfrich equation
\begin{equation} \label{helfeq}
\Delta  _ { S } H  + 2H \, (H ^2- K) - 2 \varepsilon  H  \, = \,  0 \qquad  \mbox{ in }   (-1,1).
\end{equation}
Using the calculations in \cite[Section 2.1]{DDG} in order to express the left hand side in terms of $u$ we obtain that $u$ is a solution of the following 
Dirichlet problem
\begin{eqnarray}
\lefteqn{ \frac{1}{u(x) \sqrt{1+u'(x)^2}} \, \frac{\diff }{\diff x} \bigg( \frac{u(x)}{\sqrt{1+u'(x)^2}} H'(x)\bigg) } \nonumber \\[2mm]
& & + \frac12 H(x) \bigg(\frac{u''(x)}{{(1+u'(x)^2)}^{3/2}} + \frac{1}{u(x)\sqrt{1+u'(x)^2}} \bigg)^2 - 2 \varepsilon H(x)
 = 0, \quad x\in(-1,1) \label{dirhelf}  \\[3mm]
 & & u(\pm 1)=\alpha, \; u'(\pm 1)=0.  \label{dirbc}
\end{eqnarray}
Note that (\ref{dirhelf}) is a fourth order quasilinear equation, which is elliptic, but not uniformly elliptic. \\
Let us briefly refer to three particular solutions of (\ref{dirhelf}) in the class of symmetric, positive profile curves: \\[2mm]
(i) catenoids: $u(x) = c \cosh(\frac{x}{c}) \; (c>0)$ for every $\varepsilon \geq 0$. \\[2mm]
(ii) spheres: $u(x) = \sqrt{c^2 - x^2} \; (c>1)$ in the case $\varepsilon=0$; \\[2mm]
(iii) cylinders: $u(x) = c>0$ in the case $\varepsilon = \frac{1}{4c^2}$. \\[2mm]
Catenoids and spheres have been successfully employed in the construction of symmetric Willmore surfaces (see e.g. \cite{DDG}, \cite{DFGS}) as well as in the analysis
of their asymptotic shape for small values of $\alpha$, see \cite{G}.

\section{Existence of minimisers via energy bounds} \label{subsec:exist_energy_bounds} 

\begin{lemma} \label{uniformbounds}
Let $u \in H^2(-1,1)$ be even. \\
a) If $u'(\pm 1)=0$, then $\mathscr A(u) \, \mathscr W(u) \geq 4 \pi^2$. \\[2mm]
b) If $\mathscr W(u) < 4 \pi$ and $u'(\pm 1)=0$, then 
\begin{displaymath}
\max_{x \in [-1,1]} | u'(x) | \leq \frac{\mathscr W(u)}{\sqrt{16 \pi^2 - \mathscr W(u)^2}}.
\end{displaymath}
c) If $| u'(x) | \leq M, u( \pm 1)= \alpha$, then
\begin{displaymath}
\forall x \in [-1,1]: \qquad \alpha \exp \bigl( -\frac{1}{\pi} M \sqrt{1+M^2} \, \mathscr W(u) \bigr) \leq u(x) \leq \alpha +M.
\end{displaymath}
\end{lemma}
\begin{proof} 
a) Using H\"older's inequality and (\ref{equiv})  we have
\begin{eqnarray*}
4 \pi & \leq & 2  \pi \int_{-1}^1 \underbrace{ \sqrt{ \Bigl( \frac{1}{u(x) \sqrt{1+u'(x)^2}} \Bigr)^2 +\underbrace{ \Bigl( \frac{u''(x)}{{(1+u'(x)^2)}^{3/2}} \Bigr)^2}_{\ge 0} } \, u(x) \sqrt{1+u'(x)^2} }_{\ge 1} \, \diff x \\
& \leq & \sqrt{ 2 \pi} \Bigl( \int_{-1}^1 \left[ \bigl( \frac{1}{u(x) \sqrt{1+u'(x)^2}} \bigr)^2 + \bigl( \frac{u''(x)}{{(1+u'(x)^2)}^{3/2}} \bigr)^2 \right] u(x) \sqrt{1+u'(x)^2} \,  \diff x \Bigr)^{\frac{1}{2}}
\sqrt{\mathscr A(u)}  \\
& = & \sqrt{2 \pi} \Bigl( \int_{-1}^1 \bigl( \frac{1}{u(x) \sqrt{1+u'(x)^2}} - \frac{u''(x)}{{(1+u'(x)^2)}^{3/2}} \bigr)^2 u(x) \sqrt{1+u'(x)^2} \, \diff x \Bigr)^{\frac{1}{2}}
\sqrt{\mathscr A(u)}  \\
& = & 2 \sqrt{ \mathscr W(u)} \sqrt{\mathscr A(u)}.
\end{eqnarray*}
b) Choose $x_0 \in (0,1)$ with $| u'(x_0) |= \max_{x \in [-1,1]} | u'(x) |$ and let $[a,b]=[x_0,1]$ if $u'(x_0) \geq 0$ and $[a,b]=[0,x_0]$ if
$u'(x_0)<0$. Since $u'(0)=u'(1)=0$ we calculate with the help of (\ref{equiv})
\begin{eqnarray*}
\mathscr W(u) & \geq & \pi  \int_a^b \Bigl( \frac{1}{u(x) \sqrt{1+ u'(x)^2}} - \frac{u''(x)}{{(1+ u'(x)^2)}^{3/2}} \Bigr)^2 u(x) \sqrt{1+ u'(x)^2} \, \diff x \\
& = & \pi \int_a^b \Bigl( \frac{1}{u(x) \sqrt{1+ u'(x)^2}} +  \frac{u''(x)}{{(1+ u'(x)^2)}^{3/2}} \Bigr)^2 u(x) \sqrt{1+ u'(x)^2}\,  \diff x - 4 \pi \frac{u'(x)}{\sqrt{1+ u'(x)^2}} \Big|^b_a \\
& \geq & 4 \pi \frac{| u'(x_0)|}{\sqrt{1+u'(x_0)^2}}
\end{eqnarray*}
and the required bound follows by solving for $|u'(x_0)|$. \\
c) For $x \in [0,1]$ we have
\begin{eqnarray*}
\log u(x) & = & \log u(1) - \int_x^1 \frac{u'(t)}{u(t)} \, \diff t  \geq  \log \alpha -  M \sqrt{1+M^2} \int_0^1 \frac{1}{u(t) \sqrt{1+u'(t)^2}} \, \diff t \\
& \geq &  \log \alpha - \frac{1}{\pi} M \sqrt{1+M^2} \, \mathscr W(u),
\end{eqnarray*}
which implies the lower bound on $u$. A similar idea was used in \cite[Lemma 4.9]{DFGS}. The upper bound is immediate. 
\end{proof}

\begin{corollary} \label{cor1} Suppose that $u \in H^2(-1,1)$ is even with $u'(\pm 1)=0$. Then 
$\mathscr H_\varepsilon(u) \ge 4\pi \sqrt{\varepsilon}$ and
\begin{equation} \label{wubound}
\mathscr W(u) \leq \frac{1}{2} \bigl( \mathscr H_\varepsilon(u) + \sqrt{\mathscr H_\varepsilon(u)^2 - 16 \pi^2 \varepsilon} \bigr).
\end{equation}
\end{corollary}
\begin{proof}  Lemma \ref{uniformbounds} a) implies that 
\begin{displaymath}
\mathscr H_\varepsilon(u) =  \mathscr W(u) + \varepsilon \mathscr A(u) \geq \mathscr W(u) + \frac{4 \pi^2 \varepsilon}{\mathscr W(u)}
\end{displaymath}
and hence
\begin{displaymath}
\mathscr W(u)^2 - \mathscr H_\varepsilon(u) \,  \mathscr W(u) + 4 \pi^2 \varepsilon \leq 0,
\end{displaymath}
which yields (\ref{wubound}).
\end{proof}

\begin{theorem} \label{thm:energyexistence}
Let $\alpha>0$ and $0 \leq \varepsilon \leq 4$. If $\mathscr H_\varepsilon(v)< \pi(4+\varepsilon)$ for some $v \in N_\alpha$, 
then $\mathscr H_\varepsilon$ attains a 
minimum $u \in N_\alpha$, which belongs to $C^{\infty}([-1,1])$.
\end{theorem}

\begin{proof} Let $(u_k)_{k \in \mathbb{N}} \subset N_\alpha$ be a minimising sequence such that $\mathscr H_\varepsilon(u_k) 
\searrow \inf_{w \in N_\alpha} \mathscr H_\varepsilon(w) \leq \mathscr H_\varepsilon(v)$. Corollary
\ref{cor1} implies that
\begin{displaymath}
\limsup_{k \rightarrow \infty} \mathscr W(u_k) \leq \frac{1}{2} \lim_{k \rightarrow \infty}  \bigl( \mathscr H_\varepsilon(u_k) 
+ \sqrt{\mathscr H_\varepsilon(u_k)^2 - 16 \pi^2 \varepsilon} \bigr) \leq 
\frac{1}{2} \bigl( \mathscr H_\varepsilon(v) + \sqrt{\mathscr H_\varepsilon(v)^2 - 16 \pi^2 \varepsilon} \bigr) < 4 \pi,
\end{displaymath}
since $\mathscr H_\varepsilon(v)< \pi(4+\varepsilon)$ and $\varepsilon \leq 4$. Thus there exist $\delta>0, k_0 \in \mathbb{N}$ 
such that $\mathscr W(u_k) \leq 4 \pi - \delta$ for all $k \geq k_0$. Lemma
\ref{uniformbounds} then implies that $(u_k)_{k \in \mathbb{N}}$ is bounded in $C^1([-1,1])$ and that $0<\varrho \le u_k(x)\le \frac{1}{\varrho}$ for a suitable $\varrho >0$. Then one sees from (\ref{equiv}) that a bound for ${\mathscr W} (u_k)$ also yields an $L^2$-bound for $\left(u_k''\right)_{k\in\mathbb{N}}$.
Arguing as in the proof of 
\cite[Theorem\;3.9]{DDG} we obtain a subsequence, again
denoted by $(u_k)_{k \in \mathbb{N}}$, and $u \in N_{\alpha}$ such that
\begin{displaymath}
u_k \rightharpoonup u \mbox{ in } H^2(-1,1), \quad u_k \rightarrow u \mbox{ in } C^1([-1,1])
\end{displaymath}
and it is easily seen that $u$ is a minimiser of $\mathscr H_\varepsilon$. The fact that $u$ belongs to $C^{\infty}([-1,1])$ can be shown by a straightforward
adaptation of the  corresponding argument in the proof of \cite[Theorem\;3.9]{DDG}.
\end{proof}

In order to apply Theorem \ref{thm:energyexistence} we look for suitable functions $v \in N_\alpha$ such that $\mathscr H_\varepsilon(v) < \pi(4+\varepsilon)$. 
For  the simple choice $v \equiv \alpha$ we find that 
\begin{equation} \label{energyzylinder}
\mathscr H_\varepsilon (v) = \frac{\pi}{\alpha} + 4 \pi \alpha \varepsilon.
\end{equation}
As a second useful comparison function we define (see \cite[Lemma 8.6]{GGS})
\begin{displaymath}
v(x):= 
\left\{
\begin{array}{ll}
\sqrt{r^2 -x^2}, & | x| < x_0; \\[2mm]
\alpha \cosh \bigl( \frac{|x|-1}{\alpha} \bigr), & x_0 \leq |x| \leq 1,
\end{array}
\right.
\end{displaymath}
where $r=\sqrt{x_0^2+ \alpha^2 \cosh \bigl( \frac{x_0-1}{\alpha} \bigr)^2}$ and  $x_0 \in (\frac{1}{2},1)$ is chosen in such a way that
\begin{displaymath}
x_0 + \alpha \cosh \bigl( \frac{x_0-1}{\alpha} \bigr) \sinh \bigl( \frac{x_0-1}{\alpha} \bigr)=0, \quad 
x  + \alpha \cosh \bigl( \frac{x-1}{\alpha} \bigr) \sinh \bigl( \frac{x-1}{\alpha} \bigr)>0, x \in (x_0,1].
\end{displaymath}
This condition means that the normal to the $\cosh$-function through the point $(x_0,v(x_0))$ intersects the $x$-axis in the origin.
One may observe that
$$
x\mapsto x  + \alpha \cosh \bigl( \frac{x-1}{\alpha} \bigr) \sinh \bigl( \frac{x-1}{\alpha} \bigr)
$$
is strictly increasing on $[0,1]$, strictly negative for $x=\frac12$ and strictly positive for $x=1$.
The choice of $x_0$ and $r$  ensures that $v \in N_\alpha$ and that $r \leq \sqrt{1+\alpha^2}$.  Since $H=0$ for $x_0 \leq |x| \leq 1$, 
we obtain with the help of (\ref{equiv}) that
\begin{eqnarray}
\mathscr H_\varepsilon(v) & = &  \frac{\pi}{2} \int_{-x_0}^{x_0} \Bigl( \frac{1}{ v(x)\,\sqrt{ 1+ v'(x)^2 } } 
+  \frac{v''(x)}{{(1+ v'(x)^2)}^{3/2}} \Bigr)^2 v(x) \sqrt{1+v'(x)^2} \,  \diff x - 2 \pi \frac{v'(x)}{\sqrt{1+v'(x)^2}} \Big|^{x_0}_{-x_0} 
\nonumber \\
& & +   4 \pi \varepsilon  x_0 r  + 2  \pi \alpha \varepsilon  \Bigl(1-x_0 + \frac{\alpha}{2} \sinh \bigl( \frac{2 (1-x_0)}{\alpha} \bigr) \Bigr) \nonumber \\
& <  &  4 \pi \tanh \bigl(\frac{1-x_0}{\alpha} \bigr) + \varepsilon \bigl( 4 \pi \sqrt{1+ \alpha^2} + \pi \alpha + \pi \alpha^2 \sinh \bigl( \frac{1}{\alpha} \bigr) \bigr) \nonumber  \\
& < &  4 \pi \tanh \bigl( \frac{1}{2 \alpha} \bigr) + \varepsilon \bigl( 4 \pi \sqrt{1+ \alpha^2} + \pi \alpha + \pi \alpha^2 \sinh \bigl( \frac{1}{\alpha} \bigr) \bigr), \label{energy2}
\end{eqnarray}
since $\frac{1}{2} < x_0< 1$ and 
\begin{displaymath}
\frac{1}{v(x) \sqrt{1+v'(x)^2}} + \frac{v''(x)}{{(1+v'(x)^2)}^{3/2}} = 0, \quad x \in (-x_0,x_0).
\end{displaymath}

{From} the above calculations we infer the following existence result:

\begin{theorem}
	\label{thm:existence_energy_method} 
	The functional $\mathscr H_\varepsilon$ attains a minimum on $N_\alpha$, if one of the following conditions is satisfied:
	\begin{enumerate}[label={(\roman*)}]
		\item\label{item:existence_energy_method_i} $\alpha>\frac{1}{4}$ and $0 \leq \varepsilon < \frac{1}{\alpha}$;
		\item\label{item:existence_energy_method_ii} $\alpha>0$ and $\displaystyle 0 \leq \varepsilon \leq \frac{1 - \tanh \bigl( \frac{1}{2 \alpha} \bigr)}{\sqrt{1+\alpha^2} + \frac{\alpha}{4}+ \frac{\alpha^2}{4} \sinh \bigl( \frac{1}{\alpha} \bigr) - \frac{1}{4}}$. 
	\end{enumerate}
\end{theorem}
\begin{proof} The result follows from  (\ref{energyzylinder}), (\ref{energy2}) and Theorem \ref{thm:energyexistence}.
\end{proof}

\section{Perturbation of Helfrich cylinders}
\label{sec:perturbation_cylinder}

Throughout this section we fix for a given $ \alpha >0 $
$$
{{\varepsilon}_\alpha}:=\frac{1}{4\alpha^2}
$$
as the parameter where the Helfrich cylinder
$$
u_{{{\varepsilon}_\alpha}}(x)\equiv \alpha
$$
is the unique minimiser of ${\mathscr H}_{{{\varepsilon}_\alpha}}$ (see \cite[Lemma\;4.1]{Scholtes}). Let us define
\begin{displaymath}
F(\varepsilon,u):= \Delta_{ S } H  + 2H \, (H ^2- K) - 2 \varepsilon  H.
\end{displaymath}
Recalling  \eqref{dirhelf} and taking into account that
\begin{displaymath}
H[u_{{{\varepsilon}_\alpha}}] \equiv \frac{1}{2 \alpha}, \qquad  \frac{\diff }{\diff  \delta} H[u_{{{\varepsilon}_\alpha}}+ \delta \varphi]_{| \delta =0} = \frac{1}{2} \bigl(- \frac{1}{\alpha^2} \varphi - \varphi'' \bigr)
\end{displaymath}
we calculate
\begin{eqnarray}
\lefteqn{ \langle \frac{\partial F}{\partial u}(\varepsilon_\alpha, u_{{{\varepsilon}_\alpha}}), \varphi \rangle = \frac{\diff }{\diff  \delta} F(\varepsilon_\alpha, \alpha+ \delta \varphi)_{| \delta =0} } \nonumber \\
& =  & \frac{1}{\alpha} \frac{\diff ^2}{\diff  x^2} \bigg( \alpha \frac{1}{2} \bigl(- \frac{1}{\alpha^2} \varphi - \varphi'' \bigr) \bigg) 
 + \frac{1}{4} \frac{1}{\alpha^2}   \bigl(- \frac{1}{\alpha^2} \varphi - \varphi'' \bigr) 
 +  \frac{1}{2 \alpha}  \frac{1}{\alpha} \bigl( \varphi'' - \frac{1}{\alpha^2} \varphi \bigr) -  \frac{1}{4 \alpha^2}   \bigl(- \frac{1}{\alpha^2} \varphi - \varphi'' \bigr) \nonumber \\
& = & -\frac{1}{2} \bigl( \varphi^{(iv)} + \frac{1}{\alpha^4} \varphi \bigr). \label{Ffirstvar}
\end{eqnarray}

\begin{theorem}
	\label{thm:exist_implicit_fct} 
	For every $ \alpha  > 0 $ there exists  $ \delta >0 $ such that for all $ \varepsilon  \in (  \varepsilon _ \alpha - \delta ,  \varepsilon _ \alpha + \delta ) $ the boundary value problem \eqref{dirhelf}, \eqref{dirbc} has a smooth family of solutions $ u_ \varepsilon  \in N_ \alpha \cap C^ \infty ([-1,1]) $.
\end{theorem}
\begin{proof}
	We apply the implicit function theorem as it can be found in \cite[Theorem\;15.1]{Deimling1985}. Let $ X = \mathbb{R}  $, $ U = (0, \infty ) \subset X$, $ Y = C^4( [-1,1]) $, $ V = C^4([-1,1],(0, \infty )) \subset Y$, $ Z = C^0([-1,1]) \times \mathbb{R} ^4 $. Moreover consider the function $ \Phi \colon U \times V \rightarrow Z $ defined by
	\begin{align*}
		\Phi( \varepsilon , u) \, := \, \big( F(\varepsilon,u)  \, , \, u(-1)- \alpha \, , \, u(1)- \alpha \, , \, u'(-1) \, , \, u'(1) \big) \, .
	\end{align*}
	In view of (\ref{Ffirstvar}) we obtain for $ L \colon Y \rightarrow Z $ with $\displaystyle L(\varphi):= \langle \frac{\partial \Phi}{\partial u}(\varepsilon _ \alpha,u_{{{\varepsilon}_\alpha}}), \varphi \rangle$ that
	\begin{align*}
		L (\varphi )\,=\,  \bigg( - \frac{1}{2} \bigl( \varphi ^ {(iv)}+ \frac{1}{\alpha ^4} \, \varphi \bigr) \, , \, \varphi (-1) \, , \, \varphi (1) \, , \, \varphi '(-1) \, , \, \varphi '(1) \bigg).
	\end{align*}
	We claim that $L$ is injective. To see this, consider the boundary value problem
	\begin{equation}\label{eq:helfrich_cylinder_linearised}
         \displaystyle
           \varphi ^{(iv)} +\frac{1}{\alpha^4 } \varphi 
           = 0 \mbox{ in\ } (-1,1), \quad \varphi (\pm 1) = 
           \varphi '(\pm 1) = 0.
         \end{equation}
         The roots of the characteristic equation for the differential equation in \eqref{eq:helfrich_cylinder_linearised} are given by
         $ \pm \frac{1+i}{\alpha \sqrt{ 2 } } $ and $ \pm \frac{1-i}{\alpha \sqrt{ 2 } } $ so that the functions  
         \begin{eqnarray*}
		\varphi _1(x) = \cosh \bigg( \frac{x}{\alpha \sqrt{ 2 } }\bigg) \, \cos \bigg( \frac{x}{\alpha \sqrt{ 2 } }\bigg), & & 
		\varphi _2(x) = \sinh \bigg( \frac{x}{\alpha \sqrt{ 2 } }\bigg) \, \sin \bigg( \frac{x}{\alpha \sqrt{ 2 } }\bigg), \\
		\varphi _3(x) =  \cosh \bigg( \frac{x}{\alpha \sqrt{ 2 } }\bigg) \, \sin \bigg( \frac{x}{\alpha \sqrt{ 2 } }\bigg), & & 
		\varphi _4(x) = \sinh \bigg( \frac{x}{\alpha \sqrt{ 2 } }\bigg) \, \cos \bigg( \frac{x}{\alpha \sqrt{ 2 } }\bigg) 
	\end{eqnarray*}
        form a  fundamental system. Hence there exist constants $ b _1, \dots ,b _4 \in \mathbb{R} $ such that 
	$\varphi (x) \,=\, \sum_ {j=1} ^4 b _j \, \varphi _j(x), \, x \in [-1,1]$. The four boundary conditions $\varphi(\pm 1)= \varphi'(\pm 1)=0$ translate into the linear system $M b =0$ with $b=(b_1,\ldots,b_4)$ and
	$m_{1j}=\varphi_j(1), m_{2j}=\varphi_j(-1), m_{3j}=\frac{1}{\beta} \varphi_j'(1), m_{4j}=\frac{1}{\beta} \varphi_j'(-1), j=1,\ldots,4$, where we have abbreviated
	$ \beta = \frac{1}{\alpha \sqrt{ 2 }  } > 0 $. Using the above expressions for $\varphi_j$ we find that
	\begin{align*}
		M= \begin{pmatrix}
			 \cosh (\beta ) \, \cos (\beta ) &
			 \sinh (\beta ) \, \sin (\beta ) &
			 \cosh ( \beta ) \, \sin (\beta ) &
			 \sinh ( \beta ) \, \cos (\beta ) \\
			 \cosh (\beta ) \, \cos (\beta ) &
			 \sinh (\beta ) \, \sin (\beta ) &
			 -\cosh ( \beta ) \, \sin (\beta ) &
			 -\sinh ( \beta ) \, \cos (\beta ) \\
			\begin{aligned}
				&\scriptstyle{\sinh ( \beta ) \cos ( \beta ) } \\[-0.4em] & \scriptstyle{\quad  - \cosh ( \beta )  \sin ( \beta )}
			 \end{aligned} &
			\begin{aligned}
				&\scriptstyle{\cosh ( \beta )  \sin ( \beta )) } \\[-0.4em] & \scriptstyle{\quad  + \sinh ( \beta )  \cos ( \beta )}
			 \end{aligned} &
			\begin{aligned}
				&\scriptstyle{ \sinh ( \beta )  \sin ( \beta ) } \\[-0.4em] & \scriptstyle{\quad  + \cosh ( \beta )  \cos ( \beta ) }
			 \end{aligned} &
			\begin{aligned}
				&\scriptstyle{ \cosh ( \beta )  \cos ( \beta ) } \\[-0.4em] & \scriptstyle{\quad  - \sinh ( \beta )  \sin ( \beta )  }
			 \end{aligned} \\
			\begin{aligned}
				&\scriptstyle{ -\sinh ( \beta ) \cos ( \beta ) } \\[-0.4em] & \scriptstyle{\quad  + \cosh ( \beta )  \sin ( \beta )}
			 \end{aligned} &
			\begin{aligned}
				&\scriptstyle{- \cosh ( \beta )  \sin ( \beta )) } \\[-0.4em] & \scriptstyle{\quad  - \sinh ( \beta )  \cos ( \beta )}
			 \end{aligned} &
			\begin{aligned}
				&\scriptstyle{ \sinh ( \beta )  \sin ( \beta ) } \\[-0.4em] & \scriptstyle{\quad  + \cosh ( \beta )  \cos ( \beta ) }
			 \end{aligned} &
			\begin{aligned}
				&\scriptstyle{ \cosh ( \beta )  \cos ( \beta ) } \\[-0.4em] & \scriptstyle{\quad  - \sinh ( \beta )  \sin ( \beta )  }
			 \end{aligned} 
		\end{pmatrix}
	\end{align*}
	and a calculation shows that
	\begin{align*}
		\det ( M ) \,=\, - 4 \, \sinh^2 ( \beta ) \, \cosh ^2( \beta ) + 4 \, \sin ^2( \beta ) \, \cos ^2( \beta ) \, < \, 0. 
	\end{align*}
	Hence $ b_1 = \dots = b_4 = 0 $, which yields $ \varphi  \equiv 0 $, so that $L$ is injective. Furthermore, since $ L $ allows for an elliptic theory, i.\,e. $ L $ is Fredholm of index $ 0 $, we infer that $L$ is invertible with bounded inverse. Hence we can apply the implicit function theorem to obtain solutions $ u_ \varepsilon \in N_ \alpha \cap C^4([-1,1]) $ for $ \varepsilon  $ close to $ \varepsilon _0 $. By arguing again as in the proof of Theorem \ref{thm:energyexistence} we also have $ u_ \varepsilon \in C^ \infty ([-1,1]) $.
\end{proof}


To study the behaviour of $u_\varepsilon$ for  $  \varepsilon  $ close to $  \varepsilon _ \alpha  $ we introduce 
\begin{equation} \label{defrc}
\mathit{rc}_\alpha:= \frac{\partial u_ \varepsilon }{\partial \varepsilon } |_ {\varepsilon =  \varepsilon _ \alpha }.
\end{equation}
Since $\displaystyle  \frac{\diff }{\diff  \varepsilon} F(\varepsilon,u_\varepsilon)_{| \varepsilon= \varepsilon_\alpha}=0$  we find with the help of (\ref{Ffirstvar}) that
\begin{displaymath}
0= \frac{\partial F}{\partial \varepsilon}(\varepsilon_\alpha,u_{{{\varepsilon}_\alpha}}) + \langle \frac{\partial F}{\partial u}(\varepsilon_\alpha,u_{{{\varepsilon}_\alpha}}), \mathit{rc}_\alpha \rangle = 
- \frac{1}{\alpha} -  \frac{1}{2} \bigl( \mathit{rc}_\alpha^{(iv)} + \frac{1}{\alpha^4} \mathit{rc}_\alpha \bigr)
\end{displaymath}
giving rise to the following boundary value problem for $\mathit{rc}_\alpha$:

\begin{equation}\label{eq:rate_of_change_2}
 \left\{ \begin{array}{l}
           \displaystyle
           \mathit{rc}_\alpha^{(iv)} +\frac{1}{\alpha^4 } \mathit{rc}_\alpha
           = -\frac{2}{\alpha}\quad \mbox{ in\ } \quad (-1,1),\\[4mm]
           \displaystyle \mathit{rc}_\alpha(\pm 1) = 
           \mathit{rc}_\alpha'(\pm 1) = 0.
         \end{array}
\right.
\end{equation}
This problem is solved by (cf. Figure \ref{fig:rate_of_change})
\begin{equation}\label{eq:rate_of_change_3}
 \mathit{rc}_\alpha (x) =-2\alpha^3 
 +a_\alpha \cosh\left(\frac{x}{\alpha\sqrt{2}} \right) \cos\left(\frac{x}{\alpha\sqrt{2}} \right) 
 +b_\alpha \sinh\left(\frac{x}{\alpha\sqrt{2}} \right) \sin\left(\frac{x}{\alpha\sqrt{2}} \right)
\end{equation}
with
\begin{eqnarray*}
 a_\alpha &=& \frac{2\alpha^3}{d_\alpha} \Bigg[  \sinh\left(\frac{1}{\alpha\sqrt{2}} \right)\cos\left(\frac{1}{\alpha\sqrt{2}} \right) 
 +\cosh\left(\frac{1}{\alpha\sqrt{2}} \right)\sin\left(\frac{1}{\alpha\sqrt{2}} \right) \Bigg], \\
 b_{\alpha} &=& - \frac{2\alpha^3}{d_\alpha} \Bigg[ \sinh\left(\frac{1}{\alpha\sqrt{2}} \right)\cos\left(\frac{1}{\alpha\sqrt{2}} \right) 
 -\cosh\left(\frac{1}{\alpha\sqrt{2}} \right)\sin\left(\frac{1}{\alpha\sqrt{2}} \right) \Bigg], \\
 d_\alpha & = & \cosh\left(\frac{1}{\alpha\sqrt{2}} \right) \sinh\left(\frac{1}{\alpha\sqrt{2}} \right)
 +\cos\left(\frac{1}{\alpha\sqrt{2}} \right)\sin\left(\frac{1}{\alpha\sqrt{2}} \right).
\end{eqnarray*}

Note that $d_\alpha= \frac{1}{2} \sinh(\frac{\sqrt{2}}{\alpha})+\frac{1}{2}\sin(\frac{\sqrt{2}}{\alpha})>0$. We have:

\begin{theorem}[]
	\label{thm:change_of_rate_negative} 
	Let $ \alpha >0 $ and $\mathit{rc}_\alpha$ as in (\ref{defrc}). Then $ \mathit{rc}_\alpha(x)<0  $ for all $x \in (-1,1)$.
\end{theorem}
\begin{proof}
	Let us abbreviate $\beta=\frac{1}{\alpha \sqrt{2}}$. For $ x \in [-1,1] $ we compute with the help of Corollary \ref{cor:oscillations_max} from Appendix \ref{sec:appendix_est_oscill} with $ a = \beta $
	\begin{eqnarray*}
		\frac{d_\alpha}{2 \alpha ^3} \, \mathit{rc}_\alpha (x) &= &  \,  - \bigl( \cosh(\beta)  \, \sinh(\beta) + \cos(\beta) \, \sin(\beta)\bigr) \\
		& &  \,  + \bigl( \sinh(\beta)  \, \cos(\beta) + \cosh(\beta) \, \sin(\beta)\bigr) \, \cosh(\beta x) \, \cos(\beta x)  \\
		& &  \,   - \bigl( \sinh(\beta)  \, \cos(\beta) - \cosh(\beta) \, \sin(\beta)\bigr) \, \sinh(\beta x) \, \sin(\beta x)  \\
		& \le &    - \bigl( \cosh(\beta)  \, \sinh(\beta) + \cos(\beta) \, \sin(\beta)\bigr) \\
		& &  \,  + \bigl( \sinh(\beta)  \, \cos(\beta) + \cosh(\beta) \, \sin(\beta)\bigr) \, \cosh(\beta) \, \cos(\beta)  \\
		& &  \,   - \bigl( \sinh(\beta)  \, \cos(\beta) - \cosh(\beta) \, \sin(\beta)\bigr) \, \sinh(\beta) \, \sin(\beta)  =0.
	\end{eqnarray*}
	The statement then follows since equality only holds at $ x= \pm 1 $.
\end{proof}

\begin{theorem}
	\label{thm:change_of_rate_oscillation} 
	Let  $a_c \in (\pi, \frac{3}{2} \pi)$ denote the smallest, strictly positive solution of the equation $\tanh(x)=\tan(x)$ (see Lemma \ref{lem:calc_1} in Appendix B) and $\alpha_{\operatorname{crit}} = \frac{1}{a_c \sqrt{2}}\approx 0.18008$. Then $\mathit{rc}_\alpha'(x)>0$ for all $x \in (0,1)$ and all  
	$\alpha \ge \alpha_{\operatorname{crit}}$. 
	\end{theorem}
\begin{proof}
	Let us abbreviate again $\beta=\frac{1}{\alpha \sqrt{2}}$. It follows from  Lemma \ref{lem:calc_2} in Appendix~\ref{sec:appendix_est_oscill} 
	with $ a = \beta $ that for every $x \in (0,1)$
	\begin{eqnarray*}
	\mathit{rc}_\alpha'(x) & = & (b_\alpha - a_\alpha) \cosh(\beta x) \sin(\beta x) + (b_\alpha + a_\alpha) \sinh(\beta x) \cos (\beta x) \\
	& = & \frac{2 \alpha^3}{d_\alpha} \Bigg[ \sinh(\beta x) \cos(\beta x)  \cosh(\beta) \sin(\beta) - \sinh(\beta) \cos(\beta)  \cosh(\beta x) \sin(\beta x) \Bigg] >0
	\end{eqnarray*}
since $0<\beta \leq a_c$.
\end{proof}

\begin{remark}
 The end of the proof of Lemma~\ref{lem:calc_2} shows that for $\alpha\in 
 (0,\alpha_{\operatorname{crit}})$ and $\alpha\searrow 0$, we observe an increasing number of sign changes of $\mathit{rc}_\alpha' $.
\end{remark}

\begin{figure}[h]
	\centering
	\includegraphics[]{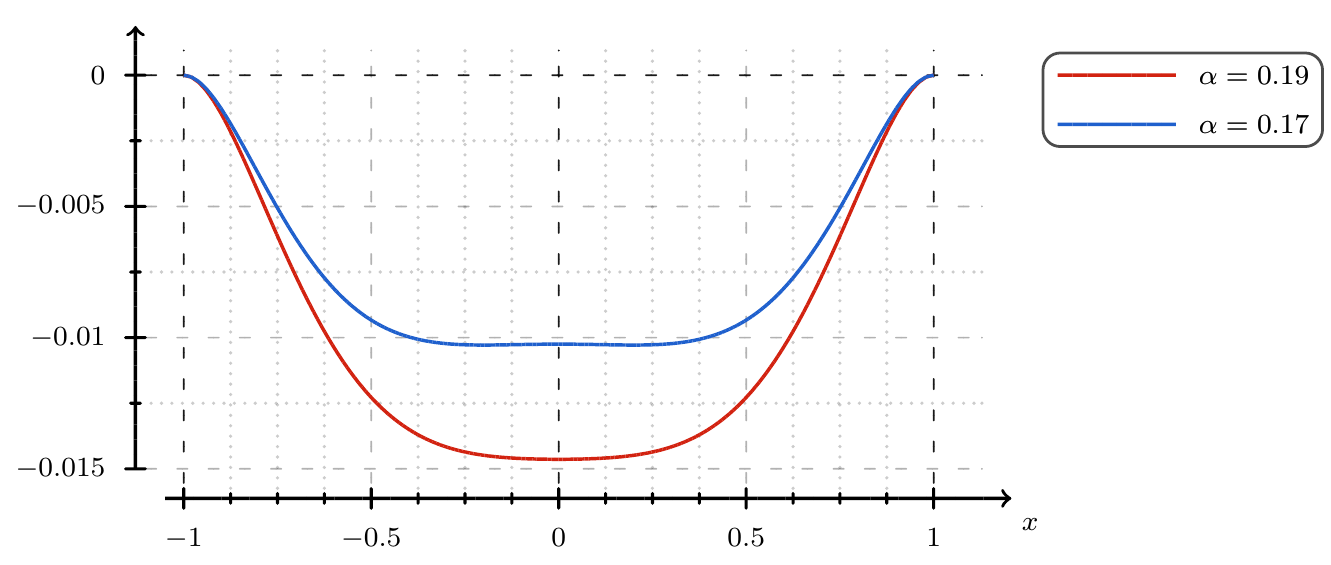}
	\caption{Plots of the rate of change function $ \mathit{rc}_\alpha $ for two different values of $ \alpha  $ around the value $\alpha_{\operatorname{crit}}\approx 0.18008$ in accordance with Theorems \ref{thm:change_of_rate_negative} and \ref{thm:change_of_rate_oscillation}. The function for $ \alpha =0.17 $ develops an oscillatory behaviour with oscillations starting at the origin $ x=0 $.}
	\label{fig:rate_of_change}
\end{figure}

\section{Existence of minimisers via gluing techniques}
\label{subsec:exist_gluing} 

We start with a discussion of symmetric positive profile curves satisfying $H \equiv 0$, i.e.
\begin{equation}  \label{minimalsurf}
\frac{1}{u(x)\, \sqrt{ 1+ u'(x)^2 } } - \frac{u''(x)}{ \big(1+u'(x)^2\big)^{3/2}} = 0 \quad \mbox{ in } [-1,1].
\end{equation}
The corresponding solutions are the catenaries
\begin{displaymath}
w_c(x):= \, c \, \cosh \bigg( \frac{x}{c } \bigg), \quad x \in [-1,1]
\end{displaymath}
with boundary values $c \cosh(\frac{1}{c})$ and surface area
\begin{align}  
	\mathscr A (w_ c ) & =\, 2 \pi \int_ {-1} ^1 w_ c( x) \, \sqrt{ 1 + {w_c'(x)}^2 } \, \diff x \nonumber\\
	&=\, 2 \pi c + \pi c^2 \, \sinh \bigg( \frac{2}{c} \bigg)
	\,=\, 2 \pi c\left( 1 + c \, \sinh \bigg( \frac{1}{c} \bigg)\, \cosh \bigg( \frac{1}{c} \bigg)\right)
	.\label{areacat}
\end{align}
A discussion of the function $c \mapsto c \cosh(\frac{1}{c})$ shows that 
\begin{equation}  \label{c0prop}
c \mapsto c \cosh(\frac{1}{c}) \mbox{ is } 
\left\{
\begin{array}{ll}
\mbox{strictly decreasing on } (0,c_0] & \\[2mm]
\mbox{strictly increasing on } [c_0,\infty), &
\end{array}
\right.
\end{equation}
where $c_0 \approx 0.8336$ is the positive solution of
the equation $c_0 =\tanh(\frac{1}{c_0})$. Setting $\alpha_0:= c_0 \cosh(\frac{1}{c_0}) \approx 1.5089$ we infer
that the equation
\begin{displaymath}
c \cosh\left(\frac{1}{c}\right) = \alpha
\end{displaymath}
has two solutions $c_1(\alpha) > c_0 > c_2(\alpha)$ if $\alpha>\alpha_0$, one solution $c_1(\alpha)$ if $\alpha=\alpha_0$ and no solution for $\alpha<\alpha_0$.   Although presumably folklore and  asymptotically obvious
for $\alpha\to \infty$, we could not  locate an easy reference for the 
following statement which is quite important in what follows:
\begin{equation}\label{eq:small_large_catenoid}
\mathscr A(w_{c_1(\alpha)}) < \mathscr A(w_{c_2(\alpha)}), \quad \alpha>\alpha_0.
\end{equation}
For the reader's convenience we give a proof in Appendix~\ref{sec:small_large_catenoid}.

In what follows we shall write
\begin{equation}  \label{defcalpha}
c_{\alpha}:=c_1(\alpha), \; \; v_{\alpha}: =w_{c_\alpha}, \; \; \hat \varepsilon_\alpha:= \frac{1}{4 c_{\alpha}^2}.
\end{equation}
In particular it follows from (\ref{c0prop}) that
\begin{equation} \label{calphamon}
(\alpha_0,\infty) \ni \alpha \mapsto c_\alpha \quad \mbox{ is strictly increasing}.
\end{equation}
In order to make the role of $v_\alpha$ in the minimisation of $\mathscr A$ more precise it is convenient to relax the class of profile curves and consider for a given boundary value $ \alpha >0 $
\begin{align*}
	\cC_ \alpha  \, := \, \Big\{ \gamma \in C^{0,1}\big([-1,1], \mathbb{R}^2\big) \, \colon \, \gamma_2(t) \ge 0 \, ,~ \gamma(\pm 1) =(\pm 1,\alpha),~ \dot \gamma (t) \not= 0 \mbox{ a.e. in } [-1,1]   \Big\},
\end{align*}
where the area functional is now given by
\begin{align*}
	\mathscr A (\gamma )  \, := \,  2 \pi \int_ {-1}^1 \gamma_2 (t) \, | \dot \gamma (t) | \, \diff t.
\end{align*}
In order to formulate the main result concerning the minimisation of $\mathscr A$ over $\cC_\alpha$ we introduce for $\alpha>0$ the Goldschmidt 
solution $\gamma_\alpha$ as a $C^{0,1}$ parametrisation of
the polygon $P_- Q_- Q_+ P_+$, where $P_{\pm}=(\pm 1,\alpha), Q_{\pm}=(\pm 1,0)$.
As a geometric object the Goldschmidt solution  corresponds to two disks with radius $\alpha$ and centers  $Q_-,Q_+$ and its surface area is given by
$\mathscr A(\gamma _ \alpha ) \,=\, 2 \pi \alpha ^2$. Appendix~\ref{sec:small_large_catenoid}   shows also  that $\mathscr A(v_{\alpha})=\mathscr A(\gamma_\alpha)$ if and only if
$\alpha = \alpha_m = c_m \cosh(\frac{1}{c_m})\approx 1.895$, where $c_m\approx 1.564$ is the unique solution of the equation
\begin{displaymath}
\frac{2}{c}   =  1 + e ^ {-2/ c}.
\end{displaymath}
For $\alpha>\alpha_m$ we have $\mathscr A(v_{\alpha})< \mathscr A(\gamma_\alpha)$ while $\mathscr A(\gamma_\alpha)< \mathscr A(v_\alpha)$ for $\alpha<\alpha_m$.
The following result is then a special case of Theorem 1 in \cite[Chapter\;8,\;Section\;4.3]{GiaquintaHildebrandt2004}:

\begin{theorem}[Absolute minimisers of the area functional] \label{thm:absoluteminarea}
	For every $\alpha>0$ the variational problem
	\begin{align*}
		\min_{\gamma \in \cC_\alpha} \mathscr A( \gamma )\,=\, 2 \pi \int_ {-1}^1 \gamma_2(t) \, | \dot \gamma (t) | \, \diff t   
	\end{align*}
	has a solution. This solution is  furnished by the catenary $t \mapsto (t,v_{\alpha}(t))$ if $\alpha> \alpha_m$, the Goldschmidt solution $\gamma_\alpha$ if $\alpha < \alpha_m$
and both of them if $\alpha=\alpha_m$. Except for reparametrisation there are no further solutions.
\end{theorem}

We use the above result in order to prove the following lemma.

\begin{lemma} \label{step1} Suppose that $\alpha \geq \alpha_m$.
For $u \in N_\alpha$ there exists $v \in N_\alpha$  such that
$v'(x)<v_{\alpha}'(x)$ for all $x \in (0,1]$ and $\mathscr H_\varepsilon(v) \leq \mathscr H_\varepsilon(u)$. If $u$ has finitely many critical points, then the same holds for $v$.
In particular we have that $v(x) > v_\alpha(x)$ for all $x \in (-1,1)$.
\end{lemma}
\begin{proof} Since $u'(1)=0<v_\alpha'(1)$, there exists $x_0 \in [0,1)$ such that 
\begin{equation} \label{ustrich}
\forall x \in (x_0,1]:\quad 
u'(x)< v_\alpha'(x)\quad \mbox{ and }  \quad u'(x_0)= v_\alpha'(x_0).
\end{equation}
If $x_0=0$ set $v=u$,  otherwise define
\begin{displaymath}
\hat c:= \inf \lbrace c \geq c_\alpha \, | \, w_c(x) > u(x) \mbox{ for all } x \in [x_0,1] \rbrace.
\end{displaymath}
Clearly,
\begin{equation}  \label{whatc}
\forall x \in [x_0,1]:\quad 
w_{\hat c}(x) \geq u(x), \quad \mbox{ and } \quad  
\exists x_1 \in [x_0,1]:\quad w_{\hat c}(x_1)=u(x_1).
\end{equation}
Suppose that $\hat c=c_\alpha$. Then, in view of (\ref{defcalpha}) and  (\ref{whatc}) we have that  $v_\alpha(x) \geq u(x)$ for $x \in [x_0,1]$ and hence $v_\alpha'(1) \leq u'(1)=0$, a contradiction, so that
$\hat c> c_\alpha$. Assume next that  $x_1 \in \lbrace x_0,1 \rbrace$. Using again (\ref{defcalpha}), (\ref{whatc}) as well as (\ref{ustrich}) we infer that 
\begin{eqnarray*}
\sinh \bigl( \frac{x_0}{c_\alpha} \bigr) = v_\alpha'(x_0) = u'(x_0) \leq w_{\hat c}'(x_0)= \sinh \bigl( \frac{x_0}{\hat c} \bigr), && \mbox{ if } x_1=x_0; \\
c_\alpha \cosh \bigl( \frac{1}{c_\alpha} \bigr)= \alpha = u(1) = w_{\hat c}(1)= \hat c \cosh \bigl( \frac{1}{\hat c} \bigr), && \mbox{ if } x_1 = 1;
\end{eqnarray*}
which is in either case a contradiction to $\hat c > c_{\alpha}$. As a result we deduce  that $x_1 \in (x_0,1)$, which implies that $w_{\hat c}'(x_1) = u'(x_1)$.
The function
\begin{displaymath}
v(x):=
\left\{
\begin{array}{ll}
u(x), & x \in [-1,-x_1] \cup [x_1,1], \\[2mm]
w_{\hat c}(x), & x \in (-x_1,x_1)
\end{array}
\right.
\end{displaymath}
\begin{figure}[h]
		\centering
		\includegraphics[]{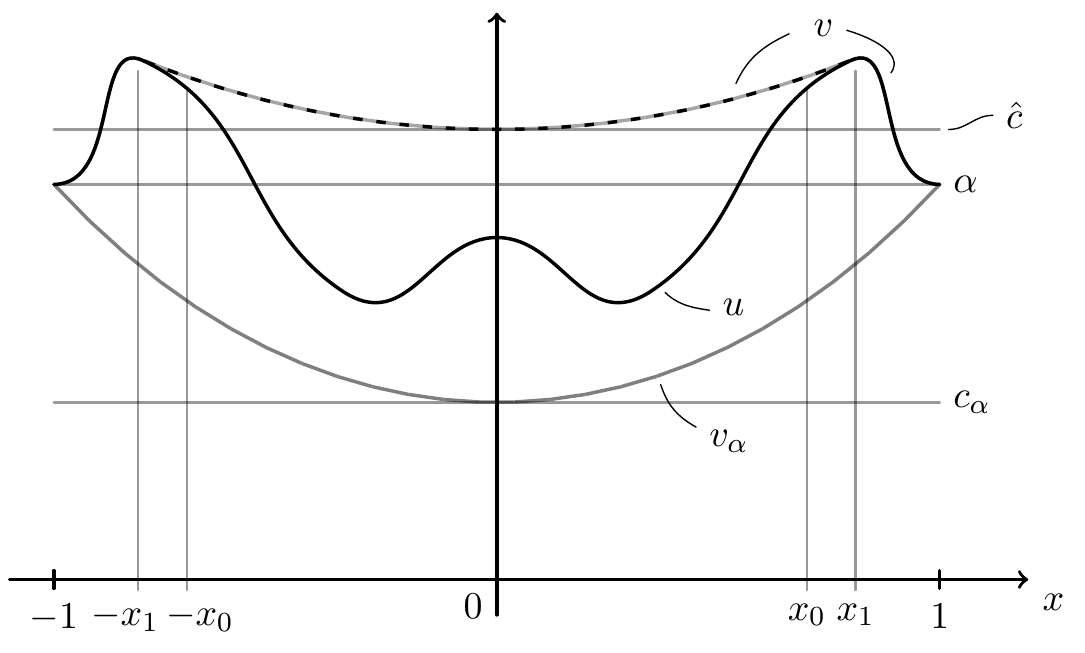}
		\caption{Construction of $v$  by attaching an appropriate catenary to $u$ at $ -x_1 $ and $ x_1 $.}
		\label{fig:existence_conv_to_cat_upper_bound_derivative}
\end{figure}
(see Figure \ref{fig:existence_conv_to_cat_upper_bound_derivative}) then belongs to $N_\alpha$. Since the catenoid given by $w_{\hat c}$ has zero mean curvature we have
\begin{eqnarray*}
\lefteqn{ \hspace{-1cm}  \mathscr H_{\varepsilon}(v) - \mathscr H_{\varepsilon}(u)  \leq 2 \pi \varepsilon \int_{-x_1}^{x_1} w_{\hat c}(x) \sqrt{1+ w_{\hat c}'(x)^2} \, \diff x - 2 \pi \varepsilon \int_{-x_1}^{x_1} u (x) \sqrt{1+ u'(x)^2} \, \diff x } \\
& = & 2 \pi \varepsilon x_1 \int_{-1}^1 w_{\hat c}(x_1 x) \sqrt{1+ w_{\hat c}'(x_1 x)^2} \, \diff x - 2 \pi  \varepsilon x_1 \int_{-1}^1 u(x_1 x) \sqrt{1+ u'(x_1 x)^2} \, \diff x \\
& = & 2 \pi  \varepsilon x_1^2 \left\{  \int_{-1}^1 w_{\hat c/x_1}(x) \sqrt{1+ w_{\hat c/x_1}'(x)^2 } \, \diff x -  \int_{-1}^1 \hat u(x) \sqrt{1+ \hat u'(x)^2 } \, \diff x \right\},
\end{eqnarray*}
where $\hat u(x)= \frac{1}{x_1} u(x_1 x), x \in [-1,1]$. Since $\frac{\hat c}{x_1} > \hat c > c_\alpha \geq c_0$, it follows from (\ref{calphamon}) that
\begin{displaymath}
\frac{\hat c}{x_1} \cosh(\frac{x_1}{\hat c}) > c_{\alpha} \cosh(\frac{1}{c_\alpha}) = \alpha \geq \alpha_m
\end{displaymath}
and Theorem \ref{thm:absoluteminarea} implies that  $w_{\hat c/x_1}$ is a global minimum for the area functional in the class $\mathscr C_{\hat \alpha}$ with $\hat \alpha=\frac{\hat c}{x_1} \cosh(\frac{x_1}{\hat c})$.
Taking into account that  
$w_{\hat c/x_1}(\pm 1)\stackrel{\mbox{\scriptsize def}}{=}\frac{\hat c}{x_1}\cosh\left(\frac{x_1}{\hat c} \right)
=\frac{1}{x_1} w_{\hat c} (x_1)
\stackrel{(\ref{whatc})}{=}\frac{1}{x_1} u(x_1)
\stackrel{\mbox{\scriptsize def}}{=}\hat u(\pm 1)
$ 
we hence infer that
\begin{displaymath}
\int_{-1}^1 w_{\hat c/x_1}(x) \sqrt{1+ w_{\hat c/x_1}'(x)^2 } \, \diff x \leq  \int_{-1}^1 \hat u(x) \sqrt{1+ \hat u'(x)^2 } \, \diff x
\end{displaymath}
giving $\mathscr H_\varepsilon(v) \leq \mathscr H_\varepsilon(u)$. We see from the construction that $v$ has finitely many critical points if this was the case for $u$. Finally, the inequality
$v(x) > v_\alpha(x), \, x \in (-1,1)$ easily follows by integration.
\end{proof}

\begin{lemma}  \label{step2} Let $\alpha \geq \alpha_m$ and $\varepsilon \geq  \hat \varepsilon_\alpha$. Suppose that
$u \in N_\alpha$ has finitely many critical points and satisfies $u'(x)<v_{\alpha}'(x)$ for all $ x \in (0,1]$. Then there exists $v \in N_\alpha$ such that 
$0 \leq v'(x) < v_{\alpha}'(x)$ for $x \in (0,1]$ and $\mathscr H_{\varepsilon}(v) \leq \mathscr H_{\varepsilon}(u)$. In particular, we have that 
$v_{\alpha}(x)< v(x) \leq \alpha$ for all $x \in (-1,1)$.
\end{lemma}
\begin{proof} Let us assume that $u'(x_0)<0$ for some $x_0 \in (0,1)$. Then there exists an interval $[a,b] \subset [0,1]$ such that
\begin{displaymath}
\forall x \in [b,1]:\quad 
u'(x) \geq 0 ;\qquad \forall x \in (a,b): \quad u'(x) < 0; \qquad u'(a)=u'(b)=0.
\end{displaymath}
Let us define
\begin{equation}  \label{defv}
v(x):= 
\left\{
\begin{array}{ll}
u(x), & x \in [-1,-b] \cup [b,1]; \\[2mm]
u(b), & x \in (-b,-a) \cup (a,b); \\[2mm]
u(x)-(u(a)-u(b)), & x \in [-a,a],
\end{array}
\right.
\end{equation}
see Figure \ref{fig:existence_conv_to_cat_lower_bound_derivative}. Clearly, $v \in H^2(-1,1)$ is even with
	\begin{figure}[h]
		\centering
		\includegraphics[]{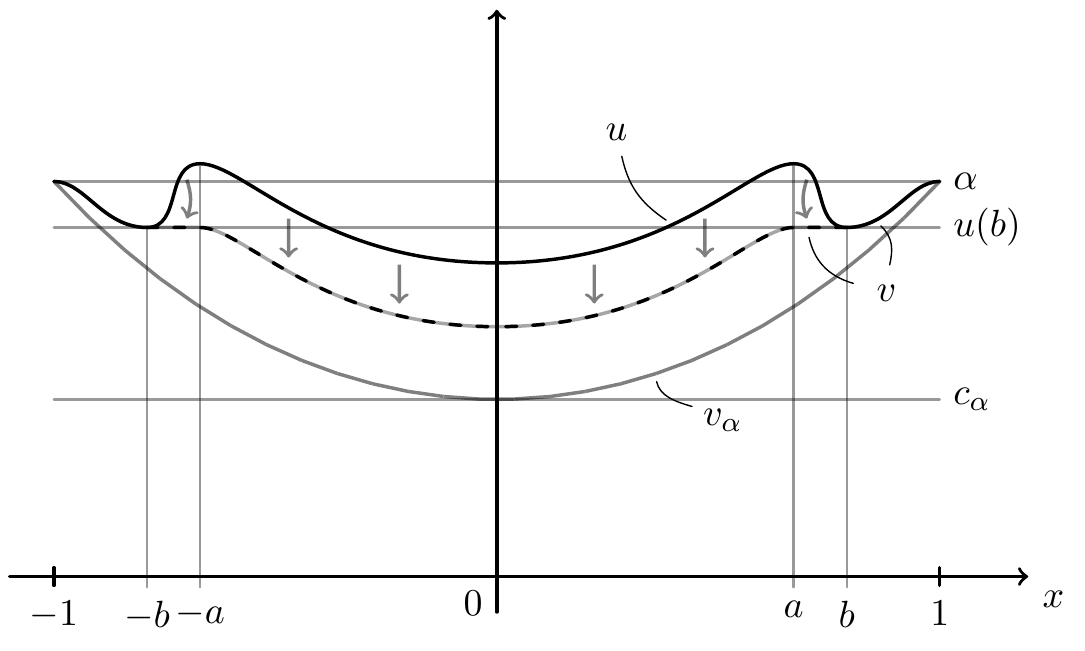}
		\caption{Construction of the function $v$ from $u$ by inserting a cylinder in the intervals $ [-b,-a], ~[a,b] $ and translating the inner part to match the cylinder.}
		\label{fig:existence_conv_to_cat_lower_bound_derivative}
	\end{figure}
\begin{equation} \label{vstrich}
v'(x)= 
\left\{
\begin{array}{ll}
u'(x), & x \in [-1,-b] \cup [b,1]; \\[2mm]
0, & x \in (-b,-a) \cup (a,b); \\[2mm]
u'(x), & x \in [-a,a],
\end{array}
\right.
\end{equation}
so that $v'(x)<v_\alpha'(x)$ for $ x \in (0,1]$ and $v'(x) \geq 0$ for $x \in [a,1]$. As a consequence,
\begin{equation}  \label{vlowerbound}
v(x) = v(1) - \int_x^1 v'(t) \, \diff t \geq \alpha - \int_x^1 v_{\alpha}'(t) \,  \diff t = v_{\alpha}(x) \geq c_\alpha , \quad x \in[0,1]
\end{equation}
and in particular $v$ is positive on $[-1,1]$, so that $v \in N_\alpha$.
Next, since $u'(x)<0$ for $x \in (a,b)$ we have that $u(x) \geq u(b)$ for $x \in [a,b]$ and therefore recalling (\ref{defcalpha})
\begin{displaymath}
\frac{1}{\sqrt{4 \varepsilon}} \leq \frac{1}{\sqrt{4 \hat \varepsilon_\alpha}} = c_{\alpha}  \leq v(x) \sqrt{1+v'(x)^2} \leq u(x) \sqrt{1+u'(x)^2}, \quad x \in [-1,1].
\end{displaymath}
Combining the above inequality with the fact that $z\mapsto 4 \varepsilon z + \frac{1}{z}$ is increasing for $z \geq \frac{1}{\sqrt{4 \varepsilon}}$  we infer that
\begin{eqnarray}
\lefteqn{ \hspace{-2cm}  4 \varepsilon v(x) \sqrt{1+v'(x)^2} + \frac{1}{v(x) \sqrt{1+v'(x)^2}}  } \nonumber  \\[2mm]
&\leq &   4 \varepsilon u(x) \sqrt{1+u'(x)^2} + \frac{1}{u(x) \sqrt{1+u'(x)^2}}, \qquad x \in [-1,1]. \label{part1} 
\end{eqnarray}
Furthermore, in view of the definition of $v$ and the fact that $v(x) \leq u(x)$ for $x \in [-b,b]$ we have
\begin{equation}  \label{part2}
\frac{v(x) v''(x)^2}{{(1+ v'(x)^2)}^{5/2}} \leq \frac{u(x) u''(x)^2}{{(1+ u'(x)^2)}^{5/2}}, \quad \mbox{ for all }x \in [-1,1]\setminus \{\pm a,\pm b \}.
\end{equation}
If we insert (\ref{part1}) and (\ref{part2}) into (\ref{eq:helfrich_functional1}) we deduce that $\mathscr H_\varepsilon(v) \leq \mathscr H_\varepsilon(u)$. 
Note that $v_{|[-a,a]}$ has less critical points than $u$. If $v'(x_1)<0$ for some $x_1 \in (0,a)$ we can repeat the above procedure on $[-a,a]$ until
we obtain a function with the desired properties in finitely many steps.
\end{proof}

\begin{remark}[]
Lemma \ref{step2} is an improvement  of \cite[Lemma\;4.8]{Scholtes} thanks to the better bound (\ref{vlowerbound}) from below. There the condition on $ \alpha  $ and $ \varepsilon  $ is given by
		\begin{align}
			\label{eq:comparision_scholtes_1} 
			\frac{4 \varepsilon \alpha \sigma \, ( \alpha - \sigma)}{4 \varepsilon \alpha ^2 +1} \,\ge\, 1
		\end{align}
		where $ \sigma  \in (0, \alpha ) $ is some appropriate constant. But equation \eqref{eq:comparision_scholtes_1} implies 
		$  
		4\le 4 \frac{4 \varepsilon \alpha  \, ( \alpha^2 /4)}{4 \varepsilon \alpha ^2 }=\alpha.
		$ 
		With the better bounds in Theorem \ref{thm:existence_conv_to_cat} we are
		able to decrease this range down to $ \alpha _m$ (and to admit arbitrarily large $\varepsilon$).
\end{remark}

\begin{theorem}[]
	\label{thm:existence_conv_to_cat}
	Assume that $ \alpha \geq \alpha _ m  $ and that $\varepsilon \geq  \hat \varepsilon _ \alpha$. Then the Helfrich functional $ \mathscr H _ \varepsilon  $ admits a minimiser $ u \in N_ \alpha$, which satisfies
    \begin{eqnarray} 
    \forall x \in (0,1): \quad 0 < u'(x) < v_{\alpha}'(x);  \label{uprop1} \\
    \forall x \in (-1,1): \quad v_\alpha(x) < u(x) < \alpha.  \label{uprop2}
    \end{eqnarray}
    The function $u$ belongs to $C^\infty ([-1,1])$ and solves the Dirichlet problem (\ref{dirhelf}), (\ref{dirbc}). 
	Furthermore, there exists a constant $c \in \mathbb{R}$ such that
\begin{equation}  \label{helf1}
\frac{u(x) u'(x) H'(x)}{1+u'(x)^2} + \frac{u(x) H(x)^2}{\sqrt{1+u'(x)^2}} - \frac{H(x)}{1+u'(x)^2} - \varepsilon \frac{u(x)}{\sqrt{1+u'(x)^2}} = c, \quad x \in [-1,1].
\end{equation}
\end{theorem}

\begin{proof} Let $ (\tilde u_k) _ {k \in \mathbb{N}} \subset N_ \alpha  $ be a minimising sequence, which we may assume to consist of functions with only finitely many critical points (e.g. polynomials). In view of Lemma \ref{step1} and Lemma \ref{step2} we may pass to a minimising sequence $(u_k)_{k \in \mathbb{N}} \subset N_\alpha$ such that 

\begin{displaymath}
\forall x \in (0,1]:  \quad  0 \leq u_k'(x)   \leq  v_{\alpha}'(x) \quad  \mbox{ and } \quad \forall x \in [-1,1]: \quad   v_{\alpha}(x) \leq u_k(x) \leq \alpha,
\end{displaymath}
so that $(u_k)_{k \in \mathbb{N}}$ is bounded in $C^1([-1,1])$. Arguing as in the proof of Theorem \ref{thm:energyexistence} we obtain the existence of a minimiser $u \in N_\alpha$ of $\mathscr H_{\varepsilon}$, which is even 
in $C^{\infty}([-1,1])$ thanks to elliptic regularity for the Euler-Lagrange equation. Furthermore, $u'(x) \geq 0$ for $x \in [0,1]$.  Since $u$ is a critical point of the Helfrich functional, it satisfies (\ref{dirhelf}), (\ref{dirbc})
as shown in Section \ref{helfrichfunctional}. Let us next denote the left hand side of (\ref{helf1}) by $M[u]$. Using Lemma \ref{noether} in Appendix \ref{sec:divergence_form_helfrich} together with (\ref{dirhelf}) 
we obtain that $\frac{\diff }{\diff x} M[u](x) = 0$ which implies (\ref{helf1}). \\
Let us next prove the strict inequality $u'(x)>0, x \in (0,1)$. If $\hat x \in [0,1]$ is a point satisfying $u'(\hat x)=0$, then (\ref{helf1})
yields
\begin{displaymath}
\frac{u(\hat x)}{4} \Bigl( \frac{1}{u(\hat x)} - u''(\hat x) \Bigr)^2 - \frac{1}{2} \Bigl( \frac{1}{u(\hat x)} - u''(\hat x) \Bigr) - \varepsilon u(\hat x) = c,
\end{displaymath}
which simplifies to
\begin{equation}  \label{hatx}
\frac{1}{4 u(\hat x)} + \varepsilon u(\hat x) = -c + \frac{ u(\hat x) u''(\hat x)^2}{4}.
\end{equation}
Let us assume that there exists $x_0 \in (0,1)$ with $u'(x_0)=0$. Since we already know that $u' \geq 0$ in $[0,1]$ we infer that $u''(x_0)=0$. Using (\ref{hatx}) for
$\hat x=0$ and $\hat x = x_0$ we obtain
\begin{displaymath}
\frac{1}{4 u(0)} + \varepsilon u(0)= -c + \frac{ u(0) u''(0)^2}{4} \geq -c = \frac{1}{4 u(x_0)} + \varepsilon u(x_0) \geq \frac{1}{4 u(0)} + \varepsilon u(0),
\end{displaymath}
where the last inequality follows since $z \mapsto \frac{1}{4z}+ \varepsilon z$ is strictly increasing for $z \geq \frac{1}{\sqrt{4 \varepsilon}}$ and  $u(x_0) \geq u(0) \geq
c_{\alpha} \geq \frac{1}{\sqrt{4 \varepsilon}}$. In particular, $u(0)=u(x_0)$ and therefore $u$ is constant on $[0,x_0]$. Standard ODE theory applied to (\ref{dirhelf}) then implies
that $u$ is constant on $[0,1]$ and hence $u \equiv \alpha$ in view of (\ref{dirbc}). This implies that $H \equiv \frac{1}{2 \alpha}$ and then again by (\ref{dirhelf})
that $\varepsilon = \frac{1}{4 \alpha^2}$. Using  the relation $\alpha = c_{\alpha} \cosh(\frac{1}{c_{\alpha}})$ we then obtain
\begin{displaymath}
\varepsilon = \frac{1}{4 \alpha^2} < \frac{1}{4 c_{\alpha}^2} =  \hat \varepsilon_\alpha,
\end{displaymath}
contradicting our assumption that $\varepsilon \geq \hat \varepsilon_\alpha$. Thus, $u'(x)>0$ for $x \in (0,1)$. Since $u$ has only finitely many critical points we may assume in view of
Lemma \ref{step1} and Lemma \ref{step2} that $u$ is a minimiser with the additional property that $0<u'(x)< v_\alpha'(x)$ for $x \in (0,1)$, which is (\ref{uprop1}). The inequalities
in (\ref{uprop2}) then follow immediately.
\end{proof}


\begin{corollary}\label{cor:existence_easy_range}
 Assume that $ \alpha \geq \alpha _ m  $. Then for any $\varepsilon \geq  0$
  the Helfrich functional $ \mathscr H _ \varepsilon  $ admits a minimiser $ u \in N_ \alpha$, which  belongs to $C^\infty ([-1,1])$ and solves the Dirichlet problem (\ref{dirhelf}), (\ref{dirbc}). 
\end{corollary}

\begin{proof}
 Since $4c>\cosh(\frac{1}{c})$ for $c\ge c_m$, we have $\frac{1}{\alpha} > \hat \varepsilon_\alpha$ for all $\alpha\ge \alpha_m$.
 Combining Theorems~\ref{thm:energyexistence} and 
 \ref{thm:existence_conv_to_cat} yields the claim.
\end{proof}

Let us next investigate the behaviour of minimisers as $\varepsilon \rightarrow \infty$.
\begin{corollary} \label{cor:existence_conv_to_cat}
Let $\alpha \geq \alpha_m$ and $(u_{\varepsilon})_{\varepsilon \geq \hat \varepsilon_\alpha}$ a sequence of minimisers of $\mathscr H_{\varepsilon}$ as obtained in Theorem \ref{thm:existence_conv_to_cat}. Then,
$u_{\varepsilon} \rightarrow v_{\alpha}$ in $W^{1,p}(-1,1)$ as $\varepsilon \rightarrow \infty$ for all $1 \leq p < \infty$.
\end{corollary}
\begin{proof} Let $(\varepsilon_k)_{k \in \mathbb{N}}$ be an  arbitrary sequence with $\varepsilon_k \geq \hat \varepsilon_\alpha$ and  $ \varepsilon _ k \rightarrow \infty, k \rightarrow \infty$. Abbreviating $u_k=u_{\varepsilon_k}$ 
we infer from (\ref{uprop1}), (\ref{uprop2}) that $(u_k)_ {k \in \mathbb{N}}$ is bounded in $C^1([-1,1])$. Thus, there exists a subsequence, again denoted by $(u_k)_{k \in \mathbb{N}}$, and
$u \in W^{1,\infty}(-1,1)$ such that
\begin{displaymath}
u_k \rightharpoonup u \mbox{ in } H^1(-1,1) \mbox{ and } u_k \rightarrow u \mbox{ in } C^0([-1,1]).
\end{displaymath}
In order to identify $u$ we claim that $u$ is a minimiser of $\mathscr A$ in the class 
\begin{displaymath}
C_\alpha= \lbrace v \in C^{0,1}([-1,1]) \, | \, v \mbox{ is even }, v>0 \mbox{ in } [-1,1],  v(\pm1)=\alpha \rbrace. 
\end{displaymath}
To see this, let $v \in C_\alpha$ and fix $0<\delta \leq \frac{1}{2} \min_{[-1,1]}v$. Since $v-\alpha \in H^1_0(-1,1)$ there exists $\zeta_{\delta} \in C^{\infty}_0(-1,1)$ (i.e. smooth and compactly supported in $(-1,1)$), which is even and satisfies $\Vert v- \alpha - \zeta_\delta \Vert_{H^1} \leq \delta$. As a result
\begin{displaymath}
\alpha + \zeta_\delta(x) \geq v(x) - \max_{[-1,1]}| v - \alpha - \zeta_\delta | \geq \min_{[-1,1]} v - \Vert v - \alpha - \zeta_\delta \Vert_{H^1} \geq \frac{1}{2} \min_{[-1,1]} v >0, \quad x \in [-1,1],
\end{displaymath}
so that
$\alpha+ \zeta_\delta \in N_\alpha$. Thus, $\mathscr H_{\varepsilon_k}(u_k) \leq \mathscr H_{\varepsilon_k}(\alpha+ \zeta_{\delta})$, which implies that
\begin{displaymath}
\mathscr A(u_k) \leq \mathscr A(\alpha+\zeta_{\delta}) + \frac{1}{\varepsilon_k} \mathscr W(\alpha+\zeta_{\delta}), \; k \in \mathbb{N}.
\end{displaymath}
Sending $k \rightarrow \infty$ we infer with the help of the  sequential weak lower semicontinuity of $\mathscr A$ in $H^1(-1,1)$ that
\begin{displaymath}
\mathscr A(u) \leq \mathscr A(\alpha+\zeta_\delta) \leq \mathscr A(v) + C \delta.
\end{displaymath}
Since $\delta$ can be chosen arbitrarily small we deduce that $\mathscr A(u)=\min_{v \in C_\alpha} \mathscr A(v)$ and hence
\begin{displaymath}
\int_{-1}^1 \varphi(x) \, \sqrt{1+u'(x)^2} \, \diff x + \int_{-1}^1 \frac{u(x) u'(x) \varphi'(x)}{\sqrt{1+ u'(x)^2}} \, \diff x = 0 \qquad \mbox{ for all } \varphi \in C^{\infty}_0(-1,1).
\end{displaymath}
Note that the above relation is first obtained for all even $\varphi \in C^{\infty}_0(-1,1)$ and then extended to arbitrary $\varphi \in C^{\infty}_0(-1,1)$ using a splitting into an
even and an odd part. Hence $u$ is a weak solution of (\ref{minimalsurf}) and it is not difficult to see that $u \in C^2([-1,1])$, so that  $u \equiv w_{c_1(\alpha)}$ or $u \equiv w_{c_2(\alpha)}$. Recalling
that $\mathscr A(w_{c_1(\alpha)}) < \mathscr A(w_{c_2(\alpha)})$ we deduce that $u = w_{c_1(\alpha)}=v_\alpha$ and hence
$(u_k)_{k \in \mathbb{N}}$ converges uniformly to $v_{\alpha}$. Furthermore, recalling that $0 \leq u_k'(x) \leq v_\alpha'(x), \, x \in [0,1]$ we have that
\begin{displaymath}
\Vert u_k' - v_\alpha' \Vert_{L^1(-1,1)} = 2 \, \int_0^1 \bigl( v_\alpha'(x) - u_k'(x) \bigr) \, \diff x = 2 \bigl( u_k(0) - v_\alpha(0) \bigr) \rightarrow 0, \, k \rightarrow \infty.
\end{displaymath}
This implies that $u_k' \rightarrow v_\alpha'$ in $L^p(-1,1)$ for all $1 \leq p < \infty$ since $(u_k')_{k \in \mathbb{N}}$ is uniformly bounded. A standard argument then shows that the whole
sequence converges to $v_\alpha$ in $W^{1,p}(-1,1)$ for every $1 \leq p < \infty$.
\end{proof}

\section{Summary \& outlook}
\label{sec:summary}

In this article we studied the Dirichlet problem (\ref{dirhelf}, \ref{dirbc}) for Helfrich surfaces of revolution depending on the parameters $ \alpha >0 $ for the boundary value and the weight factor $ \varepsilon \ge 0 $. Figure \ref{fig:domains_of_existence} in the introduction gives an overview over the existence results for solutions and in particular minimisers of the Helfrich functional.

According to Corollary~\ref{cor:existence_easy_range} we 
have the existence  of minimisers of $ \mathscr{H}_ \varepsilon  $
in $ N_ \alpha\cap C^\infty ([-1,1])  $ for all $ \alpha \ge \alpha _m $ and for all $ \varepsilon \ge 0$.
On the other hand,  with Theorem \ref{thm:existence_energy_method}\ref{item:existence_energy_method_ii} we can ensure that for every $ \alpha >0 $ and  $ \varepsilon \ge 0 $ less than a sufficiently small $ \varepsilon _0 (\alpha)$ there exists a Helfrich minimiser. Our hope was to be able to prove existence of minimisers also for $ \alpha < \alpha _ m $ and for all $ \varepsilon \ge 0 $. This turned out to be a hard task (too hard for us)
because minimisers tend to develop oscillations being created around $ x=0 $ and therefore gluing techniques do not seem to work. A first attempt to study this domain of parameters was done in Section \ref{sec:perturbation_cylinder}. In particular Theorem~\ref{thm:change_of_rate_oscillation}, the analysis in Appendix \ref{sec:appendix_est_oscill} and Figure \ref{fig:rate_of_change} show that for $ \alpha < \alpha_{\operatorname{crit}}\approx 0.18008 $ oscillatory behaviour of the rate of change function occurs, leading us to anticipate a similar behaviour for minimisers in the vicinity of the Helfrich cylinders. We conjecture that for any $ \alpha < \alpha _m $ there will be some $ \tilde \varepsilon _ \alpha   $ such that for $ \varepsilon >  \tilde \varepsilon _ \alpha   $ 
(except possibly $\varepsilon _ \alpha $) all minimisers show oscillatory behaviour around the cylinder $x\mapsto \frac{1}{2\sqrt{\varepsilon}}$. Theorem~\ref{thm:change_of_rate_oscillation} and Lemma~\ref{lem:calc_2} 
indicate that for $\alpha\in (0,\alpha_{\operatorname{crit}})$ one may expect 
 that even $  \tilde \varepsilon _ \alpha  < \varepsilon _ \alpha  $. This 
 expectation is supported by numerical experiments.

\label{0} 
The benefit of using gluing techniques like those we used in Section \ref{subsec:exist_gluing} is that we obtain qualitative results for the minimisers. In Theorem \ref{thm:existence_conv_to_cat} we found that for $ \alpha > \alpha _m $ and $ \varepsilon > \hat \varepsilon _ {\alpha } $ the minimisers lie below the Helfrich cylinder, i.\,e. its boundary value at $ \alpha  $. We expect this to be true for all $ \alpha >0 $ and $ \varepsilon > \varepsilon _ \alpha  $. On the other hand we do not know anything for the regime $ 0<\varepsilon < \varepsilon _ \alpha   $ other than the linearisation results from Section \ref{sec:perturbation_cylinder} which indicate that minimisers start to grow as $ \varepsilon  $ decreases. Thus we expect for these minimisers to lie above the Helfrich cylinder and possibly below the  Willmore minimisers. \\ \-

To conclude, let us collect some open problems:
\begin{enumerate}[label={(\arabic*)}]
    \item  $ \alpha \ge \alpha _m $: In Corollary~\ref{cor:existence_conv_to_cat},
    do we have smooth convergence away from the boundary as
    $ \varepsilon \nearrow \infty  $, i.e. in any 
    $C^m_{\operatorname{loc}} (-1,1)$?
	\item $ \alpha < \alpha _m $: Existence of minimisers for all $ \varepsilon >0 $?
	\item $ \alpha < \alpha _m $: Analogously to Corollary~\ref{cor:existence_conv_to_cat}, does the sequence of minimisers $ (u_ \varepsilon )_ {\varepsilon >0} $ converge to the Goldschmidt solution as $ \varepsilon \nearrow \infty  $? In which sense?
	\item $ \alpha < \alpha _m $: What happens in the boundary layer as $ \varepsilon \nearrow \infty  $?
	\item $ 0 < \varepsilon < \varepsilon _ \alpha   $: Do minimisers lie below the Willmore minimisers and above the Helfrich cylinder, i.\,e. $x\mapsto  \alpha  $?
	\item $ \varepsilon > \varepsilon _ \alpha   $: Do minimisers lie below the Helfrich cylinder, i.\,e. $ x\mapsto \alpha  $?
	\item Are there parameter intervals for $\alpha$ such that the 
	values $ u_ \varepsilon (x) $ of (suitable) Helfrich minimisers for $ x \in [-1,1] $ are a decreasing function in $ \varepsilon  $?
	\item Can one quantify the domain where oscillatory behaviour occurs?
	\item Can one expect uniqueness of minimisers, at least in some parameter regime? To our knowledge this question is completely open.
\end{enumerate}

\begin{appendix}



\section{The ``small'' and the ``large'' catenoid}
\label{sec:small_large_catenoid} 

For the reader's convenience we collect some arguments, which do not only show inequality (\ref{eq:small_large_catenoid}) but also the remarks before Theorem~\ref{thm:absoluteminarea}. We do not claim originality, but we could not locate a reference for what follows.

We consider $c\in(0,\infty)$ as independent variable and use as 
before $w_c$ for the corresponding catenoids and the radii 
of their boundary circles 
$$
(0,\infty ) \ni c \mapsto \alpha (c):=c\cosh\left(\frac{1}{c}\right)
$$
as a function of $c$. We calculate the first three derivatives:
\begin{eqnarray*}
	\alpha'(c) &=& \cosh\left(\frac{1}{c}\right)-\frac{1}{c} \sinh\left(\frac{1}{c}\right),\\
	\alpha''(c) &=& \frac{1}{c^3} \cosh\left(\frac{1}{c}\right)>0,\\
	\alpha'''(c) &=& -\frac{3}{c^4} \cosh\left(\frac{1}{c}\right)
	                 -\frac{1}{c^5} \sinh\left(\frac{1}{c}\right)<0.
\end{eqnarray*}	
As remarked above, $\alpha'(c)=0 \Leftrightarrow c=\tanh\left(\frac{1}{c}\right)  \Leftrightarrow c=c_0$.
The strict monotonicity of $\alpha''$ yields
$$
\forall c\in(0,c_0):\quad \alpha''(c_0-c) > \alpha''(c_0+c).
$$
Using $\alpha'(c_0)=0$, integration yields
$
\forall c\in(0,c_0):\quad -\alpha'(c_0-c) > \alpha'(c_0+c).
$
After a further integration we come up with
\begin{equation}\label{eq:convexity_radii_catenoids}
\forall c\in(0,c_0):\quad \alpha(c_0-c) > \alpha(c_0+c).
\end{equation}
The idea in proving (\ref{eq:small_large_catenoid}) consists 
in comparing the area of the catenoids $w_c$ and of the 
Goldschmidt solutions $\gamma_{\alpha(c)}$:
\begin{eqnarray*}
	\frac{1}{2\pi}\left({\mathscr A}(w_c)-{\mathscr A}(\gamma_{\alpha(c)}) \right) &=& c+\frac{c^2}{2} \sinh\left(\frac{2}{c}\right)
	-c^2  \cosh^2\left(\frac{1}{c}\right)\\
	&=& c-\frac{c^2}{2}-\frac{c^2}{2} e^{-2/c}.
\end{eqnarray*}
We hence define the function
$$
(0,\infty ) \ni c \mapsto g (c):=c-\frac{c^2}{2}-\frac{c^2}{2} e^{-2/c}.
$$
One may note that $g$ may be smoothly extended to $[0,\infty ) $.
As above we calculate the first three derivatives:
\begin{eqnarray*}
g'(c) &=& 1-c-c  e^{-2/c}- e^{-2/c},\\
g''(c) &=& -1- e^{-2/c}-\frac{2}{c}  e^{-2/c}-\frac{2}{c^2}  e^{-2/c}<0,\\
g'''(c) &=& -\frac{4}{c^4} e^{-2/c}<0.
\end{eqnarray*}	
We observe that $g(\, .\, )$ has the same unique critical point as $\alpha(\, .\, )$, i.e.
$g'(c)=0\Leftrightarrow c=c_0$,
$$
g'(c)>0\Leftrightarrow c<c_0 \quad \mbox{ and } g'(c)<0\Leftrightarrow c>c_0. 
$$
As for $\alpha$ we conclude from the strict monotonicity of the 
second derivatives and from $g'(c_0)=0$ that 
\begin{equation}\label{eq:convexity_area_catenoids}
\forall c\in(0,c_0):\quad g(c_0-c) > g(c_0+c).
\end{equation}

We may now prove (\ref{eq:small_large_catenoid}) and take any $\alpha >\alpha_0$. As above we choose
$$
0<c_2 (\alpha) < c_0 < c_1(\alpha) \quad \mbox{ with } \quad 
\alpha =\alpha( c_1(\alpha)) =\alpha( c_2(\alpha)) .
$$	
Inequality~(\ref{eq:convexity_radii_catenoids}) yields
$$
\alpha( c_1(\alpha)) =\alpha( c_2(\alpha))=\alpha(c_0-(c_0-c_2(\alpha)))
>\alpha(c_0+(c_0-c_2(\alpha)))=\alpha (2c_0-c_2(\alpha)).
$$
Since $\alpha(\, .\, )$ is strictly \emph{increasing} on $(c_0,\infty)$,
this yields 
$$
c_1(\alpha) > 2 c_0 -c_2(\alpha).
$$
On the other hand, $g(\, .\, )$ is strictly \emph{decreasing} on $(c_0,\infty)$, and we find by means of (\ref{eq:convexity_area_catenoids})
$$
g( c_1(\alpha)) <g( 2 c_0 -c_2(\alpha))=g(c_0+(c_0-c_2(\alpha)))
<g(c_0-(c_0-c_2(\alpha)))=g(c_2(\alpha)).
$$
Since the Goldschmidt contribution is the same for $c_1(\alpha)$ 
and $c_2(\alpha)$, we come up with 
$$
{\mathscr A}(w_{c_1(\alpha)})<{\mathscr A}(w_{c_2(\alpha)}),
$$
as claimed. Finally one should note that $g(c)<0\Leftrightarrow c>c_m$, which shows that in this regime, the ``small'' catenoid $w_{c_1(\alpha)}$ has smaller area than the Goldschmidt 
solution as well as the ``large'' catenoid $w_{c_2(\alpha)}$.


\section{Divergence form of the Helfrich equation}
\label{sec:divergence_form_helfrich} 

We note that the form of $M[u]$ in the case $\varepsilon=0$ is derived in \cite[Theorem 3]{DADW} and is a consequence of 
the invariance of the Willmore functional with respect to 
translations. This means that the Helfrich equation can be written in divergence form. 
For the Willmore equation this was observed and exploited 
by Rusu \cite{Rusu}, Rivi\`{e}re \cite{Riviere} and many others.

\begin{lemma} \label{noether}  For $u \in C^4([-1,1])$ let 
\begin{displaymath}
M[u](x):= \frac{u(x) u'(x) H'(x)}{1+u'(x)^2} + \frac{u(x) H(x)^2}{\sqrt{1+u'(x)^2}} - \frac{H(x)}{1+u'(x)^2} - \varepsilon \frac{u(x)}{\sqrt{1+u'(x)^2}}.
\end{displaymath}
Then we have
\begin{eqnarray*}
\frac{\diff }{\diff x} M[u](x) & = & u(x) u'(x) \left\{ \frac{1}{u(x) \sqrt{1+u'(x)^2}} \, \frac{\diff }{\diff x} \bigg( \frac{u(x)}{\sqrt{1+u'(x)^2}} H'(x)\bigg) \right.  \\[2mm]
& & + \left.  \frac12 H(x) \bigg(\frac{u''(x)}{{(1+u'(x)^2)}^{3/2}} + \frac{1}{u(x)\sqrt{1+u'(x)^2}} \bigg)^2 - 2 \varepsilon H(x) \right\}.
\end{eqnarray*}
\end{lemma}
\begin{proof} A straightforward calculation shows that
\begin{eqnarray*}
\frac{\diff }{\diff x} \Bigl( \frac{u u' H'}{1+ (u')^2} \Bigr) & = &  \frac{u'}{\sqrt{1+(u')^2}} \, \frac{\diff }{\diff x} \Bigl( \frac{u H'}{\sqrt{1+ (u')^2}} \Bigr) + \frac{u u'' H'}{{(1+(u')^2)}^2}; \\
\frac{\diff }{\diff x} \Bigl( \frac{u H^2}{\sqrt{1+(u')^2}} \Bigr) & = & \frac{u' H^2}{\sqrt{1+(u')^2}}  + 2 \frac{uH H'}{\sqrt{1+(u')^2}} - \frac{u u' u'' H^2}{{(1+(u')^2)}^{3/2}}; \\
\frac{\diff }{\diff x} \Bigl( \frac{H}{1+(u')^2} \Bigr) & = & \frac{H'}{1+ (u')^2} - 2 \frac{u' u'' H}{{(1+ (u')^2)}^2}; \\
\frac{\diff }{\diff x} \Bigl( \frac{u}{\sqrt{1+ (u')^2}} \Bigr) & = & 2 u u' H.
\end{eqnarray*}
Combining the above relations and taking into account that
\begin{displaymath}
\frac{u u''}{{(1+(u')^2)}^2} + 2 \frac{uH}{\sqrt{1+(u')^2}} - \frac{1}{1+ (u')^2} =0
\end{displaymath}
we obtain
\begin{eqnarray*}
\lefteqn{ \frac{\diff }{\diff x} M[u] = u u' \left\{ \frac{1}{u \sqrt{1+(u')^2}} \, \frac{\diff }{\diff x} \Bigl( \frac{u H'}{\sqrt{1+ (u')^2}} \Bigr)  \right. } \\
& &  \left. + H \Bigl( \frac{H}{u \sqrt{1+(u')^2}}  - \frac{u'' H}{{(1+(u')^2)}^{3/2}} + 2 \frac{u''}{u {(1+(u')^2)}^2} \Bigr)  - 2 \varepsilon H  \right\} \\
& = & u u' \left\{ \frac{1}{u \sqrt{1+(u')^2}}\,  \frac{\diff }{\diff x} \Bigl( \frac{u H'}{\sqrt{1+ (u')^2}} \Bigr) + \frac{1}{2} H \Bigl( \frac{u''}{{(1+ (u')^2)}^{3/2}} + \frac{1}{u \sqrt{1+(u')^2}} \Bigr)^2 -  2 \varepsilon H
\right\}.
\end{eqnarray*}
\end{proof}

\section{Estimating the oscillations}
\label{sec:appendix_est_oscill} 

In order to prove the auxiliary results used in Section \ref{sec:perturbation_cylinder} we will give a description of the oscillatory behaviour of the functions
\begin{align*}
	x \, &\mapsto \, A \, \cosh (ax) \, \cos (ax) - B \, \sinh (ax) \, \sin (ax) \, , \\ 
	x \, &\mapsto \, A \, \sinh (ax) \, \cos (ax) - B \, \cosh (ax) \, \sin (ax) \, , 
\end{align*}
where $ A,B $ and $ a $ are given constants.

\begin{lemma}[]
	\label{lem:oscillations_monotonicity_aux}
	Let $ K \not= 0 $ be a given constant and consider the function
	\begin{displaymath}
		\eta  \colon [0,1)  \rightarrow [0,\infty) \, , \qquad \eta (y) \, := \,  1 - \frac{(1+K^2) \, y^2}{1 + K^2 \, y^2}  \, .
	\end{displaymath}
	Then  for all $ y \in [0,1) $ we have $ \eta (y) >0 $. Furthermore: 
	\begin{itemize}
	\item[(a)] If $K < -1$, then there exists $y_\ast \in (0,\frac{1}{\sqrt{2}})$ such that $1+K \eta(y_\ast)=0$ and 
	$1 + K \, \eta (y)<0$ on $(0,y_ \ast)  $, $1 + K \, \eta (y)>0$  on $(y _ \ast,1) $. 
	\item[(b)] If $K \geq -1$, then $1+ K \eta(y)>0$ for all $y \in (0,1)$.
	\end{itemize}
\end{lemma}

\begin{proof}
	Rewriting 
	$$
	\eta (y) =\frac{1-y^2}{1+K^2y^2}=\frac{1+K^2}{K^2} \cdot \frac{1}{1+K^2y^2} - \frac{1}{K^2}
	$$
	shows that $\eta (\, .\,)$ is strictly decreasing  and hence that $\eta(y)>0$ for $y \in [0,1)$. If
	$K<-1$ a  calculation shows that 
	\begin{displaymath}
	1+ K \eta(y)= \frac{K+1}{1+ K^2 y^2} \Bigl( 1+ \frac{K-1}{1+\frac{1}{K}} \, y^2 \Bigr),
	\end{displaymath}
	which implies (a) with $y_ \ast = \bigl( \frac{1 + \frac{1}{K} }{1 - K} \bigr)^{\frac{1}{2}}$. Note that $y_\ast \in (0, \frac{1}{\sqrt{2}})$, since $K<-1$.
	The assertion (b) is straightforward.
	\end{proof}

\begin{proposition}
	\label{prop:oscillations_extrema} 
	Consider constants $ a \not= 0 $ and $ A,B $ which do not vanish simultaneously, i.\,e. $ A \not= 0 $ or $ B \not= 0 $. Define the even function
	\begin{align*}
		h \colon \mathbb{R} \rightarrow \mathbb{R} \, ,	\qquad h(x) \, := \, A \, \cosh (ax) \, \cos (ax) - B \, \sinh (ax) \, \sin (ax) \, .
	\end{align*}
	Then the local extrema of $ h $ are isolated and strictly increasing in their absolute values on the positive real axis, i.\,e.
	\begin{align*}
		\big| h(x_1)\big| \,<\,\big| h(x_2)\big| \qquad \text{for local extrema} \quad  0 \le  x_1 < x_2 \, .
	\end{align*}
\end{proposition}
\begin{proof}
	Without loss of generality we  assume that $ a = 1 $ since $ h $ is symmetric and the statement is invariant under linear transformations. 
	Since the case where $A=0$ or $B=0$ is quite simple we may assume that $A\not=0$ and $B\not=0$.
	Further we can restrict ourselves to the case $ A > 0 $; otherwise  $- h$ is considered. Setting $ A=1 $ would also be admissible 
	but we like to keep the ``$ A $'' to highlight the connection to $ B $ in the following discussion.  
	Because $h''(x)=-2A \sinh (x) \, \sin (x)-2B \cosh (x) \, \cos (x)$ we see that 
	$h$ has also a strict local minimum in $0$ iff $B<0$ and a strict local maximum  in $0$ iff $B\ge0$. In the first case $h$ must have a further extremum 
	before its first zero.\\ \-
	
	First of all we rewrite
	\begin{align*}
		h(x) \, = \, A \, \cosh (x) \, \cos (x) - B \, \sinh (x) \, \sin (x)  \,=\, E(x) \, \cos \big( x + \varphi (x)\big) \, ,
	\end{align*}
	with
	\begin{align*}
		E(x) \, := \, \big[ A^2 \, \cosh ^2(x) + B^2 \, \sinh ^2(x)\big] ^ {\frac{1}{2} } 
	\end{align*}
	and $ \varphi  $ is the smooth angular map defined modulo $ 2 \pi  $ by
	\begin{align*}
		\cos \big( \varphi (x)\big) \,=\, \frac{A}{E(x)} \, \cosh (x) \qquad \text{and} \qquad \sin \big( \varphi (x)\big)
		\,=\, \frac{B}{E(x)} \, \sinh (x) \, .
	\end{align*}
	We choose $ \varphi  $ such that it is uniquely determined by $ \varphi (0)=0 $, i.\,e. $ \varphi (x) \in (- \pi/2 , \pi /2) $. 
	If $B>0$ then $ \varphi (x) \in [0 , \pi /2) $, if $B<0$ then $ \varphi (x) \in (- \pi/2 , 0]$, and if $B=0$, then $ \varphi (x)\equiv 0$.
	For the derivative of $ \varphi  $ we compute
	\begin{align*}
		\cos \big( \varphi (x)\big)  \, \varphi '(x) \,&=\, \frac{\diff}{\diff x} \Big( \sin \big( \varphi (x)\big)  \Big)  \, =\, \frac{B}{E(x)} \, \cosh (x) - \frac{B}{E(x)^2} \, \sinh (x) \, E'(x) \\
		&= \, \frac{B}{E(x)} \, \cosh (x) - \frac{B (A^2+B^2)}{E(x)^3} \, \sinh ^2(x) \, \cosh (x) \, ,
	\end{align*}
	which yields 
	\begin{align}
		\label{eq:oscillations_extrema_phi_derivative} 
		\varphi '(x) \, = \, \frac{B}{A} \, \bigg[ 1 - \frac{(A^2+B^2) \, \sinh ^2(x)}{A^2 \, \cosh ^2(x) + B^2 \, \sinh ^2(x)} \bigg]  \,=\, \frac{B}{A} \, \Bigg[ 1 - \frac{ \big( 1 + \frac{B^2}{A^2} \big)  \, \tanh ^2(x)}{1 + \frac{B^2}{A^2}  \, \tanh ^2(x)} \Bigg] \, .
	\end{align}
	Putting $K=\frac{B}{A}$ and $y=\tanh(x)$ we apply Lemma~\ref{lem:oscillations_monotonicity_aux} to $\eta (y):= \frac{1}{K} \varphi'(x)$.
	Observing that $\frac{\diff }{\diff x}(x+ \varphi(x))=1+K\eta (y)$ we see that $\frac{\diff }{\diff x}(x+ \varphi(x))$ is either always positive 
	or negative first and then positive. Because $x+ \varphi(x)>-\frac{\pi}{2}$, $\operatorname{Id} + \varphi$ has to become strictly 
	increasing before the first zero of $h$. Denoting $0<x_1<\ldots<x_k<\ldots$ the zeroes of $h$ this shows:
	$$
	x_k+ \varphi(x_k)=\left(k-\frac12 \right)\pi.
	$$
	In particular
	\begin{equation}\tag{C} \label{Claim:C}
	x_k\in 
	\begin{cases}
	\left( \left(k-1 \right)\pi,\left(k-\frac12 \right)\pi\right) \quad & \mbox{\ if\ } B> 0,\\
	 \left[ \left(k-\frac12 \right)\pi, k\pi\right) \quad &\mbox{\ if\ } B\le 0.\\
	\end{cases}
	\end{equation}
	Moreover, $|\cos (x+ \varphi(x))|$ attains in each interval $(x_k,x_{k+1})$ the value $1$.

	Next we  also rewrite the derivative of $ h $
	\begin{align*}
		h'(x) \, = \, (A-B) \, \sinh (x) \, \cos (x) - (A+B) \, \cosh (x) \, \sin (x)  \,=\, F(x) \, \cos \big( x + \psi (x)\big) \, ,
	\end{align*}
	with
	\begin{align*}
		F(x) \, := \, \big[ (A+B)^2 \, \cosh ^2(x) + (A-B)^2 \,  \sinh ^2(x)\big] ^ {\frac{1}{2} } 
	\end{align*}
	and the smooth function $ \psi  $ determined by
	\begin{align*}
		\cos \big( \psi (x)\big) \,=\, \frac{A-B}{F(x)} \, \sinh (x) \qquad \text{and} \qquad \sin \big( \psi (x)\big) \,=\, \frac{A+B}{F(x)} \, \cosh (x) \, .
	\end{align*}
	Here we choose $ \psi (0) = \frac{\pi }{2}  $ in case $ -A < B $, $ \psi (0)=0 $ if $ -B=A $ and $ \psi (0) = - \frac{\pi }{2}  $ else, such that $ \psi (x) \in (- \pi , \pi ) $. Similar to $ \varphi  $ we have
	\begin{align}
		\label{eq:oscillations_extrema_psi_derivative}
		\begin{aligned}
		\psi '(x) \,&=\, - \frac{A-B}{A+B} \, \Bigg[ 1 - \frac{\big( (A+B)^2+(A-B)^2\big) \, \sinh ^2(x)}{(A+B)^2 \, \cosh ^2(x) + (A-B)^2 \, \sinh ^2(x)} \Bigg]  \\
		&=\, \frac{B-A}{A+B}  \, \left[ 1 - \frac{ \Big( 1 + \frac{(B-A)^2}{(A+B)^2} \Big)  \, \tanh ^2(x)}{1 + \frac{(B-A)^2}{(A+B)^2}  	\, \tanh ^2(x)} \right] \, .
		\end{aligned}
	\end{align}
	Let $0=y_0<y_1<\ldots<y_k<\ldots$ denote the zeroes of $h'$.

	Now, since $ E $ and $ F $ are strictly increasing, the proposition follows if we are able to show the claim:
	\begin{align*}
		\label{eq:oscillations_extrema_claim}\tag{D}
		{\parbox[c]{32em}{\centering \itshape Between two consecutive zeroes of $ h $ in $ (0, \infty ) $ there is exactly one zero of $ h' $ and vice versa.}}
	\end{align*}
	Claims (\ref{Claim:C}) and (\ref{eq:oscillations_extrema_claim}) yield the proof of the proposition,
	because:
	\begin{itemize}
	 \item Case $B\ge 0$, $0$ is a local maximum of $h$:\\
	 We shall see that $0=y_0<x_1<\ldots < y_k < x_{k+1}< y_{k+1}<\ldots$ and the above claims yield:
	 $$
	 \sup_{x\in [x_{k-1},x_k]} |h(x)|=
	 |h(y_{k-1})| <\sup_{x\in [x_{k-1},x_k]} E(x)\le \inf_{x\in [x_{k},x_{k+1}]} E(x)<|h(y_k)|
	 =\sup_{x\in [x_{k},x_{k+1}]} |h(x)|.
	 $$
	 The strict inequalities above are due to the strict monotonicity of $E$ and to the facts that $h(x_k)=0$ and that $|\cos(\,.\,)|=1$ once on each $[x_k,x_{k+1}]$.
	 \item Case $B < 0$, $0$ is a local minimum of $h$:
	 We shall see that $0=y_0<y_1 <x_1<\ldots < y_k < x_{k}< y_{k+1}<\ldots$ and the above claims yield:
	 $$
	 \sup_{x\in [x_{k-1},x_k]} |h(x)|=
	 |h(y_k)| <\sup_{x\in [x_{k-1},x_k]} E(x)\le \inf_{x\in [x_{k},x_{k+1}]} E(x)<|h(y_{k+1})|
	 =\sup_{x\in [x_{k},x_{k+1}]} |h(x)|.
	 $$ 
	\end{itemize}
	Note that the local extremum at $ x=0 $ has always the smallest absolute value among all local extrema. \\ \-

	To prove the claim we need to discuss several cases (see Figure~\ref{fig:cases_phi_psi}):
	\begin{figure}[h]
		\centering
		\includegraphics[]{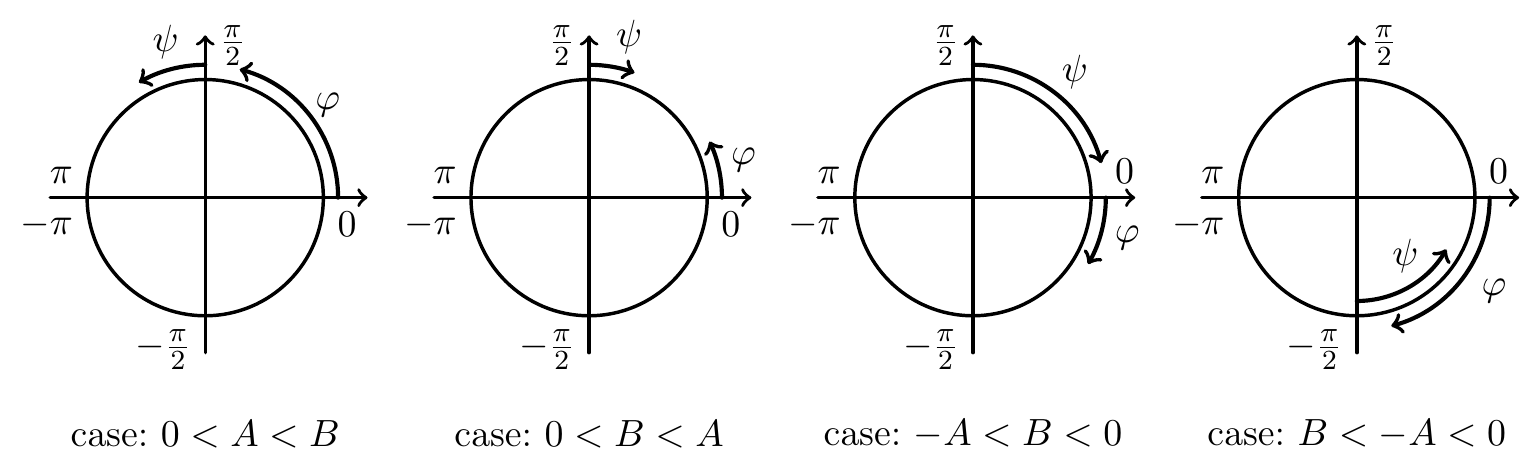}
		\caption{Sketch of the qualitative behaviour of $ \varphi (x) $ and $ \psi (x) $ as $ x $ is increased in four different cases. 
		Here the cases $ B=A $, $ B=0 $ and $ B=-A $ are omitted.}
		\label{fig:cases_phi_psi}
	\end{figure}
	~\\ \-

	\noindent{\itshape Case $ 0< A \le B $:} Combining \eqref{eq:oscillations_extrema_phi_derivative} and 
	\eqref{eq:oscillations_extrema_psi_derivative} with Lemma \ref{lem:oscillations_monotonicity_aux} by 
	setting $ K= \frac{B}{A}  $, resp. $ K = \frac{B-A}{A+B}  $, and $ y = \tanh (x) $ we see that $ \varphi  $ 
	and $ \psi  $ are strictly increasing with
	\begin{align*}
		\varphi (0) \,&=\, 0 \, , & \varphi (\infty ) \, := \, \lim _ {x \rightarrow \infty } \varphi (x) \,
		&=\, \arcsin \bigg( \frac{B}{\sqrt{ A^2 + B^2 }  } \bigg) \, \in \, \Big( 0, \frac{\pi }{2} \Big)  \, , \\
		\psi (0) \,&=\, \frac{\pi }{2} \, ,  & \psi (\infty ) \, := \, \lim _ {x \rightarrow \infty } \psi (x) \, 
		&= \, \arccos \Bigg( \frac{A-B}{\sqrt{ (A+B)^2 + (A-B)^2 }  } \Bigg) \, \in \, \Big[ \frac{\pi }{2}, \pi \Big) \, . 
	\end{align*}
	This immediately implies for all $ x > 0 $
	\begin{align}
		\label{eq:oscillations_extrema_separated}
		\varphi (x) \, < \, \psi (x) \, < \, \varphi  (x) + \pi \, , \qquad \text{i.\,e.} \qquad x + \varphi (x) \, < \, x+ \psi (x) \, < \,  x+ \varphi  (x) + \pi \, .
	\end{align}
	Since  $  \mathrm{Id} + \psi  $ is strictly increasing we have that $y_k+\psi(y_k)=(k+1/2)\pi$ and further 
	\begin{eqnarray*}
	 x_k+\psi(x_k)&<&x_k+\varphi (x_k)+\pi= \left( k+\frac12\right)\pi\\
	 &=&y_k+\psi(y_k) =x_{k+1}+\varphi(x_{k+1}) < x_{k+1}+\psi(x_{k+1}).
	 .
	\end{eqnarray*}
	Hence, again by strict monotonicity of $  \mathrm{Id} + \psi  $
	$$
	x_k<y_k<x_{k+1},
	$$
	which proves Claim \eqref{eq:oscillations_extrema_claim}. \\ \-

	\noindent{\itshape Case $ 0 \le  B < A $:} Applying Lemma~\ref{lem:oscillations_monotonicity_aux} once again 
	we see that $ \varphi  $ is still strictly increasing, but $ \psi  $ is strictly decreasing with
	\begin{align*}
		&\varphi (0) \,=\, 0 \, , \quad \varphi (\infty ) \, \in \,   \Big[ 0, \frac{\pi }{2} \Big)  \, , \\ 
		&\psi (0) \, = \,  \frac{\pi }{2} \, , \quad  \psi (\infty )=\arcsin\left(\frac{A+B}{\sqrt{(A+B)^2+(A-B)^2}} \right) 
		\, \in \,  \Big( 0,\frac{\pi }{2} \Big) \, . 
	\end{align*}
	This implies for all $ x>0 $ that $ | \psi (x) - \varphi (x)| < \frac{\pi }{2} $. Moreover, since $B < A$
	\begin{displaymath}
	 \varphi ( \infty ) =  \arcsin \bigg( \frac{B}{\sqrt{ A^2 + B^2 }  } \bigg) \,  < \, \arcsin \Bigg( \frac{A+B}{\sqrt{ (A+B)^2 + (A-B)^2 }  } \bigg) = \psi ( \infty ),
	\end{displaymath}
	and \eqref{eq:oscillations_extrema_separated} still holds. 
	Obviously $ \mathrm{Id} + \varphi  $ is strictly increasing. But as $ \frac{B-A}{A+B} \ge -1 $ we also get 
	from Lemma \ref{lem:oscillations_monotonicity_aux} that $ \mathrm{Id} + \psi $ is strictly  increasing
	so that $y_k+\psi(y_k)=(k+1/2)\pi$ as above. Like in the previous case we obtain again the validity of 
	Claim \eqref{eq:oscillations_extrema_claim}. \\ \-

	\noindent{\itshape Case $ -A<B<0 $:} $ \varphi  $ and $ \psi  $ are strictly decreasing with 
	\begin{align*}
		\varphi (0) \,=\, 0 \, , \quad \varphi (\infty ) \,  \in \,   \Big( - \frac{\pi }{2} ,0 \Big)  \, , \qquad \psi (0) \,
		= \,  \frac{\pi }{2} \, , \quad  \psi (\infty ) \, \in \,  \Big( 0,\frac{\pi }{2} \Big) \, . 
	\end{align*}
	Similar to the first case this immediately implies \eqref{eq:oscillations_extrema_separated}. The strict monotonicity of
	$ \mathrm{Id}  + \varphi  $ follows from Lemma \ref{lem:oscillations_monotonicity_aux} like in the previous step as 
	$ \frac{B}{A} > -1 $. {From} Lemma \ref{lem:oscillations_monotonicity_aux} with $ \frac{B-A}{A+B} < -1 $ we also know 
	that $ \mathrm{Id} + \psi  $ possesses a zero at $ x_ \ast < \tanh ^{-1}\big( 1/ \sqrt{ 2 }\big) \approx 0.8814  $ 
	so that for $ x > x_ \ast  $ it is strictly increasing and decreasing for $ x \in (0, x_ \ast ) $. 
	In the latter situation we have
	\begin{align*}
		&- \frac{\pi }{2} \,<\, 0 + \varphi (0) \,<\, x + \varphi ( x  ) \,<\,x_ \ast + \varphi (x_ \ast )\,<\, x_ \ast + \psi (x_ \ast ) \,\le\,x + \psi (x) \,<\, 0 + \psi (0) \,=\, \frac{\pi }{2} \, ,
	\end{align*}
	so that neither $ h  $ nor $ h' $ have a zero in $ (0, x_ \ast ] $. 
	Hence $y_1>x_ \ast$ and
	$$
	y_0+\psi(y_0)=\frac{\pi}{2},\qquad \forall k\in\mathbb{N}:\quad y_k+\psi(y_k)=\left(k-\frac12 \right)\pi.
	$$
	For $k\in\mathbb{N}$ we obtain using the monotonicity behaviour of $ \mathrm{Id}  + \varphi  $ and $ \mathrm{Id} + \psi  $:
	\begin{eqnarray*}
	 y_k+\varphi(y_k) < y_k+\psi(y_k) =\left(k-\frac12 \right)\pi = x_k +\varphi(x_k) &\Rightarrow& y_k<x_k,\\
	 y_{k+1}+\psi(y_{k+1})=\left(k+\frac12 \right)\pi = x_k +\varphi(x_k) +\pi >  x_k +\psi(x_k)
	 &\Rightarrow& x_k < y_{k+1}.
	\end{eqnarray*}
	We conclude that $0=y_0<y_1<x_1<\ldots <y_k<x_k< y_{k+1}<\ldots$, which is 
	again  Claim \eqref{eq:oscillations_extrema_claim}. \\ \-

	\noindent{\itshape Case $ B \le -A<0 $:} We only prove the claim for $ B < -A $. The case $ B= -A $ follows from a similar discussion. 
	Here $ \varphi  $ is strictly decreasing and $ \psi  $ is strictly increasing with 
	\begin{align*}
		\varphi (0) \,=\, 0 \, , \quad \varphi (\infty ) \, \in \,   \Big( - \frac{\pi }{2} ,0 \Big)  \, , 
		\qquad \psi (0) \,=\,  -\frac{\pi }{2} \, , \quad  \psi (\infty ) \, \in \,  \Big( -\frac{\pi }{2} , 0\Big) \, . 
	\end{align*}
	We first would like to establish the separation property \eqref{eq:oscillations_extrema_separated}. Obviously $ | \varphi (x) - \psi (x) | < \frac{\pi }{2}    $ holds for all $ x > 0 $. Unfortunately $ \varphi  $ and $ \psi  $ will intersect in a point $  x_ \ast  > 0 $ so that \eqref{eq:oscillations_extrema_separated} will only hold for $ x >  x_ \ast  $. For $ x > 0 $ we consider the ``distance'' between $ \varphi  $ and $ \psi  $ given by
	\begin{align}
		G(x)^2 \, &:= \, \Big[ \cos  \big( \varphi (x)\big) - \cos  \big(  \psi (x)\big)  \Big] ^2 + \Big[ \sin \big( \varphi (x)\big) - \sin \big( \psi (x)\big)  \Big] ^2 \nonumber \\
		&\phantom{:} = \, \cos ^2 \big( \varphi (x)\big) - 2 \, \cos \big( \varphi (x)\big) \, \cos \big( \psi (x)\big)   + \cos ^2 \big( \psi (x)\big) \nonumber  \\
		&\mathrel{\phantom{=}} \, + \sin ^2 \big( \varphi (x)\big)  - 2 \, \sin \big( \varphi (x)\big) \, \sin \big( \psi (x)\big)  + \sin ^2 \big( \psi (x)\big) \label{G}  \\
		&\phantom{:} =  \, 2 - 2 \, \cos \big( \varphi (x)\big) \, \cos \big( \psi (x)\big) - 2 \, \sin \big( \varphi (x)\big) \, \sin \big( \psi (x)\big) \nonumber  \\
		&\phantom{:} = \, 2 - \frac{2  A \, (A-B)}{E(x) \, F(x)} \, \cosh (x) \, \sinh (x) - \frac{2  B \, (A+B)}{E(x) \, F(x)} \, \sinh (x) \, \cosh (x) \nonumber  \\
		&\phantom{:} = \, 2 - 2 \, \frac{(A^2 + B^2) \, \cosh (x) \, \sinh (x)}{\big[ A^2 \, \cosh ^2(x) + B^2 \, \sinh ^2(x)\big] ^ {1/2} \, \big[ (A+B)^2 \, \cosh ^2(x) + (A-B)^2 \, \sinh ^2(x)\big] ^ {1/2 } } \nonumber   \\
		&\phantom{:} = \, 2 -  2 \,  \,\frac{(A^2 + B^2) \, \tanh (x)}{\big[ A^2 + B^2 \, \tanh ^2(x)\big] ^ {1/2} \, \big[ (A+B)^2 + (A-B)^2 \, \tanh ^2(x)\big] ^ {1/2 } }  \nonumber .
	\end{align}
	If we substitute $y=\tanh(x)$ we find that $G(x)^2=0$ if
	\begin{displaymath}
	(A-B)^2 B^2 y^4 - 2 A B (A^2 - B^2) y^2 + A^2 (A+B)^2 =0,
	\end{displaymath}
	whose  positive solution is given by $y_\ast=\bigl( \frac{A (A+B)}{B(A-B)} \bigr)^ { \frac{1}{2} }$. Recalling that $B < -A<0$ we find 
	\begin{align*}
	 y_\ast \,=\, \Bigg( \frac{A \big( | B | -A\big)}{| B | \big(A+ | B | \big)} \Bigg)^ { \frac{1}{2} } \, < \, \Bigg( \frac{A | B | }{| B | \big(A+ | B | \big)} \Bigg)^ { \frac{1}{2} } \,<\,  \frac{1}{\sqrt{ 2 } }.   
	\end{align*}
	Thus $ \varphi  $ and $ \psi  $ will intersect in the point $ x_ \ast = \tanh^{-1}(y_\ast)  < \tanh ^{-1}\big( 1/ \sqrt{ 2 }\big) \approx 0.8814 $. {From} the strict monotonicity of $ \varphi  $ and $ \psi  $ we then get
	that $ \psi (x) < \varphi (x) $ for $ x < x_ \ast  $ and $ \psi (x) > \varphi (x) $ for $ x > x_ \ast  $. 
	Since $\psi<0$ and $\varphi<0$ we find that $x_1>\frac{\pi}{2}, y_1>\frac{\pi}{2}$ and 
	$$
	\forall k\in\mathbb{N}_0:\quad y_k+\psi(y_k)=\left(k-\frac12\right)\pi.
	$$
	As in the previous case we see that $ \mathrm{Id} + \varphi  $ is strictly increasing on $[\pi/2,\infty)$ and $ \mathrm{Id} + \psi  $ 
	on $[0,\infty)$. This yields for $k\in\mathbb{N}$:
	\begin{eqnarray*}
	 y_k+\psi(y_k)  =\left(k-\frac12 \right)\pi = x_k +\varphi(x_k)< x_k +\psi(x_k)&\Rightarrow& y_k<x_k,\\
	 y_{k+1}+\psi(y_{k+1})=\left(k+\frac12 \right)\pi = x_k +\varphi(x_k) +\pi >  x_k +\psi(x_k)
	 &\Rightarrow& x_k < y_{k+1}.
	\end{eqnarray*}
	We conclude that $0=y_0<y_1<x_1<\ldots <y_k<x_k< y_{k+1}<\ldots$, so that again
	 Claim \eqref{eq:oscillations_extrema_claim} holds true.
\end{proof}

\begin{remark}
	The proof of Proposition \ref{prop:oscillations_extrema} also shows, that we always have
	\begin{align*}
		\lim _ {x \rightarrow \infty } \big( \psi (x) - \varphi (x)\big) \,=\, \frac{\pi }{4} \, .
	\end{align*}
	In fact, recalling the function $G$ in (\ref{G}) we obtain
	\begin{displaymath}
		\lim _ {x \rightarrow \infty } G (x) = 2 - 2 \frac{A^2 + B^2}{ \big[ A^2 + B^2\big] ^ {1/2} \, \big[ 2A^2 + 2B^2\big] ^ {1/2} } 
		=  2 - \sqrt{2}\, 
	\end{displaymath}
	which  shows that the points $\exp(i\psi(\infty))$ and $\exp(i\varphi(\infty))$ 
	on the unit circle are $\sqrt{2-\sqrt{2}}$ apart from each other. The claim now follows from $2\arcsin(\sqrt{2-\sqrt{2}}/2)=\pi/4$.
\end{remark}

\begin{corollary}
	\label{cor:oscillations_max} 
	Let $ a > 0 $ be given. Then the function
	\begin{align*}
		[0,1] \ni x \, \mapsto \, &\big[ \sinh (a) \, \cos (a) + \cosh (a) \, \sin (a)\big] \, \cosh (ax) \, \cos (ax)  \\
		& \qquad  - \big[ \sinh (a) \, \cos (a) - \cosh (a) \, \sin (a)\big] \, \sinh (ax) \, \sin (ax)
	\end{align*}
	attains its maximum at $ x=1 $.
\end{corollary}
\begin{proof}
	Here we apply Proposition \ref{prop:oscillations_extrema} with 
	\begin{align*}
		A \, &:= \, \sinh (a) \, \cos (a) + \cosh (a) \, \sin (a) \, , \\
		B \, &:= \, \sinh (a) \, \cos (a) - \cosh (a) \, \sin (a) \, .
	\end{align*}
	Note that we can exclude the case $ A = B = 0 $. Otherwise we would have $ \cos (a) = \sin (a) = 0 $ which is not possible. 
	Let $ h $ denote the function from Proposition \ref{prop:oscillations_extrema} with our choice for $ A $, $ B $ and $ a $. 
	In what follows we show, that $ x=1 $ is a local maximum of $ h $. Therefore we compute
	\begin{align*}
		h(x) \, &= \, A \, \cosh (ax) \, \cos (ax) - B  \, \sinh (ax) \, \sin (ax) \\
		&= \, \big[ \sinh (a) \, \cos (a) + \cosh (a) \, \sin (a)\big] \, \cosh (ax) \, \cos (ax) \\
		&\mathrel{\phantom{=}} \, - \big[ \sinh (a) \, \cos (a) - \cosh (a) \, \sin (a)\big] \, \sinh (ax) \, \sin (ax) \, ,\\[0.5em]
		\frac{1}{a} \, h'(x) \, &= \,  (A-B) \, \sinh (ax) \, \cos (ax) - (A + B) \, \cosh (ax) \, \sin (ax)  \\
		&= \, 2 \, \cosh (a) \, \sin (a) \, \sinh (ax) \, \cos (ax) - 2 \, \sinh (a) \, \cos (a) \, \cosh (ax) \, \sin (ax) \, , \\[0.5em]
		\frac{1}{a^2} \, h''(x) \, &= \,  - 2A \, \sinh (ax) \, \sin (ax) -2B \, \cosh (ax) \, \cos (ax) \\
		&= \, - 2 \, \big[ \sinh (a) \, \cos (a) + \cosh (a) \, \sin (a)\big] \, \sinh (ax) \, \sin (ax) \\
		&\mathrel{\phantom{=}} \,   -2 \, \big[ \sinh (a) \, \cos (a) - \cosh (a) \, \sin (a)\big] \, \cosh (ax) \, \cos (ax) \, .
	\intertext{We conclude  that}
		h(1) \, &= \, \sinh (a) \, \cosh (a) \, \cos ^2(a) + \cosh ^2(a) \, \sin (a) \, \cos (a) \\
		&\mathrel{\phantom{=}} \, - \sinh ^2(a) \, \sin (a) \, \cos (a) + \sinh (a) \, \cosh (a) \, \sin ^2(a)  \\
		&= \, \sinh (a) \, \cosh (a) + \sin (a) \, \cos (a) \\
		& > \, 0 \, ,\\[0.5em]
		\frac{1}{a}\, h'(1) \, &= \, 2 \, \cosh (a) \, \sin (a) \, \sinh (a) \, \cos (a) - 2 \, \sinh (a) \, \cos (a) \, \cosh (a) \, \sin (a) \\
		&= \, 0 \, ,\\[0.5em]
		\frac{1}{a^2}\, h''(1) \, &= \, -2 \, \sinh ^2(a) \, \sin (a) \, \cos (a) - 2 \, \sinh (a) \, \cosh (a) \, \sin ^2(a) \\
		&\mathrel{\phantom{=}} \, -2 \,  \sinh (a) \, \cosh (a) \, \cos ^2(a)  + 2 \, \cosh^2 (a) \, \sin (a) \, \cos (a) \\
		&= \, - 2 \, \sinh (a) \, \cosh (a) + 2 \, \sin (a) \, \cos (a) \\
		& < \, 0 \, .
	\end{align*}
	{From} Proposition \ref{prop:oscillations_extrema} we conclude, that $ h(1) $ as a local maximum satisfies 
	\begin{align*}
		h(1) \,=\, \big| h(1)\big| > \big| h(x) \big| \,\ge\, h(x) \, ,
	\end{align*}
	where $ x \in [0,1) 
	$ is an arbitrary local extremum. In particular this must also hold not only for $ x $ being a local extremum, but for any arbitrary point in $ [0,1 ) 
	$ which proves the claim.
\end{proof}

\begin{lemma}\label{lem:calc_1}
We consider the function
$$
f:D_f:= [0,\infty)\setminus \pi \mathbb{N} \to \mathbb{R},\quad f(x):=\frac{\tanh(x)}{\tan(x)}.
$$
This function is strictly decreasing on each connected component of its domain of definition,
$$
\forall x\in D_f\setminus \{0\}: f'(x) <0.
$$
Since $f(0)=1$ and $\lim_{x\to k\pi\pm 0} f(x)=\pm\infty$ there exists a smallest strictly positive solution 
$a_c\in (\pi,\frac{3}{2}\pi)$ of the equation $f(x)=1$, i.e.
$$
\tan(a_c)=\tanh(a_c),\quad \forall x\in (0,\pi):\quad  \frac{\tanh(x)}{\tan(x)}<1,\quad 
\forall x\in (\pi,a_c):\quad \frac{\tanh(x)}{\tan(x)}>1.
$$
\begin{figure}[h] 
\centering 
\includegraphics{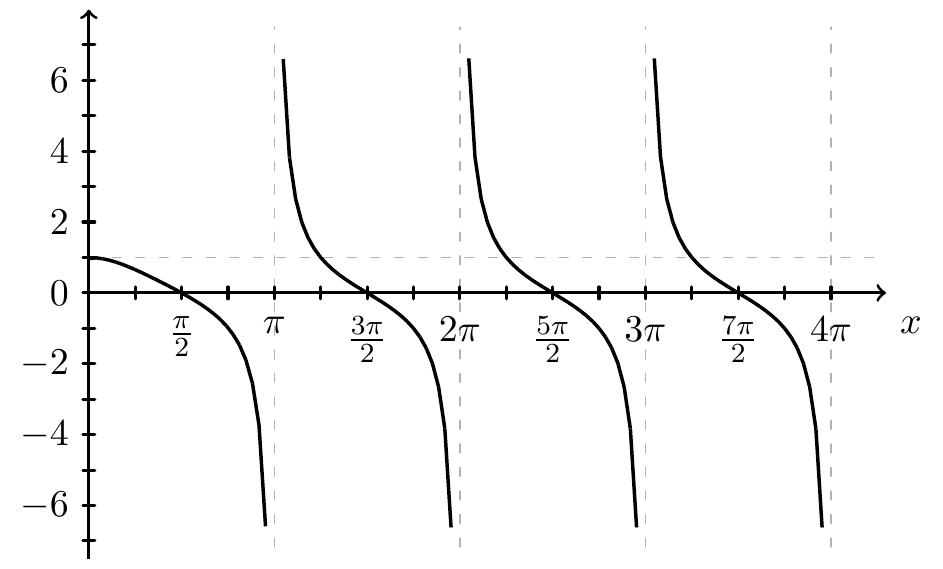}
\caption{Graph of $x\mapsto \frac{\tanh(x)}{\tan(x)}$.}
\label{fig:appendix_1} 
\end{figure} 
\end{lemma}

\begin{proof}
We calculate for $x\in D_f\setminus \{0\}$:
\begin{eqnarray*}
 f(x) &=& \frac{\sinh(x)\cos (x)}{\cosh(x)\sin (x)} \\
 f'(x) &=& \frac{(\cosh(x)\cos (x) -\sinh(x)\sin (x))\cdot \cosh(x)\sin (x)-\sinh(x)\cos (x)\cdot (\sinh(x)\sin (x)+\cosh(x)\cos (x))  }{\cosh^2(x)\sin^2 (x)}\\
 &=& \frac{\sin (x)\cos (x)\cdot (\cosh^2(x)-\sinh^2(x) )
 -\sinh (x)\cosh (x)\cdot (\sin^2 (x) + \cos^2 (x))}{\cosh^2(x)\sin^2 (x)}\\
 &=& \frac{\sin (x)\cos (x) -\sinh (x)\cosh (x)}{\cosh^2(x)\sin^2 (x)}
 = \frac{\sin (2x) -\sinh (2x)}{2\cosh^2(x)\sin^2 (x)}<0.
\end{eqnarray*}

\end{proof}

\begin{lemma}\label{lem:calc_2}
Let $a_c\in (\pi,\frac{3}{2}\pi))$ be as in Lemma~\ref{lem:calc_1}. Then for $a>0$ we have that
$$
\forall x\in (0,a):\quad 
\cosh(a)\sin(a) \sinh(x)\cos(x) - \sinh(a)\cos(a) \cosh(x)\sin(x) >0,
$$
iff
$$
a\in (0,a_c].
$$
\end{lemma}

 \begin{proof}
  The claim is obvious for $a\in \pi \mathbb{N}$, so we only need to consider $a\in (0,\infty)\setminus \pi \mathbb{N}$.
  We consider the function $f$ as in Lemma~\ref{lem:calc_1}.
  
  \medskip 
  \noindent 
  {\it Case $a\in (0,\pi)$:}\\
  Since $f$ is strictly decreasing on $(0,a)$, we have for all $x\in (0,a)$:
  \begin{eqnarray}
  f(x) & > & f(a)\nonumber \\
  \Rightarrow \quad \frac{\tanh(x)}{\tan(x)} & > & \frac{\tanh(a)}{\tan(a)}\nonumber\\
  \Rightarrow \quad \sinh(x) \cos(x) \cosh(a) \sin(a) & > & \sinh(a) \cos(a )\cosh(x) \sin(x).\label{eq:calc_1}
  \end{eqnarray}
  In the last step we used that $\sin(x)>0$, $\sin(a)>0$.
  
  \medskip 
   \noindent 
  {\it Case $a\in (\pi,a_c]$:}\\
  Here $f(a)\ge f(a_c)=1$. We consider first $x\in (\pi, a)$. Thanks to Lemma~\ref{lem:calc_1} we start with
  $$
  f(x)  >  f(a)\quad \Rightarrow \quad \frac{\tanh(x)}{\tan(x)} >\frac{\tanh(a)}{\tan(a)}. 
  $$
  Since in this case $\sin(x)<0$, $\sin(a)<0$,  we end up again with (\ref{eq:calc_1}).
  For $x\in (0,\pi)$, the starting point is
  $$
  f(x)  <1=f(a_c)\le  f(a)\quad \Rightarrow \quad \frac{\tanh(x)}{\tan(x)} <\frac{\tanh(a)}{\tan(a)}.
  $$
  But since now $\sin(x)>0$, $\sin(a)<0$,  (\ref{eq:calc_1}) follows again. For $x=\pi$, the claim is obvious.
  
    \medskip 
    \noindent 
  {\it Case $a > a_c$:}\\
  According to Lemma~\ref{lem:calc_1} and the definition of $a_c$, $f(0,a)=\mathbb{R}$ so that we always find 
  $x_1,x_2\in (0,a)$ with
  $$
  \frac{\tanh(x_1)}{\tan(x_1)} >\frac{\tanh(a)}{\tan(a)}> \frac{\tanh(x_2)}{\tan(x_2)}
  $$
  and both $x_1, x_2$ in the same interval $(k\pi,(k+1)\pi)$. So  the sign of $\sin(x_1)$ and $\sin(x_2)$ coincides.
  Depending on the sign of $\sin(a)$ one chooses $x_0=x_1$ or $x_0=x_2$ respectively and ends up
  in each case with a point $x_0\in (0,a)$ for which (\ref{eq:calc_1}) is violated, while in the other point (\ref{eq:calc_1}) is satisfied.
 \end{proof}

\end{appendix}

\bigskip\noindent 
{\bf Acknowledgement.} We are grateful to the referee for his or her very careful reading of the manuscript and for his or her very helpful suggestions.

\end{document}